\documentclass[11pt]{article}
 \usepackage{benstyle}
 
\DefineNamedColor{named}{forestgreen} {rgb}{0.13,0.54,0.13}

\DefineNamedColor{named}{oceanfoam} {rgb}{0.34, 0.52, 0.54}

\DefineNamedColor{named}{burgundy}{rgb}{0.61, 0, 0}

\usepackage{cite}

\title{Optimal approximation of infinite-dimensional holomorphic functions }
\author{Ben Adcock\thanks{Department of Mathematics, Simon Fraser University, Burnaby BC, Canada} \and Nick Dexter\thanks{Department of Scientific Computing, Florida State University, Tallahassee, USA} \and Sebastian Moraga\footnotemark[1]}

\begin{document}

\maketitle


\begin{abstract}
Over the several decades, approximating functions in infinite dimensions {from samples} {has} gained increasing attention in {computational science and engineering}, especially in computational uncertainty quantification. This is primarily due to the relevance of functions that are solutions to parametric differential equations in various fields, e.g. chemistry, economics, engineering, and physics.  While acquiring accurate and reliable approximations of such functions is inherently difficult, current benchmark methods exploit the fact that such functions often belong to certain classes of holomorphic functions to get algebraic convergence rates in infinite dimensions with respect to the number of {(potentially adaptive)} samples $m$. Our work focuses on providing theoretical {approximation} guarantees for the class of so-called \textit{$(\bm{b},\varepsilon)$-holomorphic} functions, {demonstrating} that these algebraic rates are the best possible for Banach-valued functions in infinite dimensions. {We} establish lower bounds using a reduction to a discrete problem {in combination with the theory of} $m$-widths, Gelfand widths and Kolmogorov widths. We study two cases, \emph{known} and \emph{unknown anisotropy}, in which the relative importance of the variables is known and unknown, respectively. A key conclusion of our paper is that in the latter setting, approximation from finite samples is impossible without some inherent ordering of the variables, even if the samples are chosen adaptively. Finally, in both cases, we demonstrate near-optimal, non-adaptive (random) sampling and recovery strategies which achieve {close to} same rates as the lower bounds. 
\end{abstract}


\pbk

\pbk
\textbf{Corresponding author:} {\tt smoragas@sfu.ca} 

\section{Introduction}

Approximating infinite-dimensional Banach-valued functions is an essential task in {computational science and engineering}. Different physical processes are studied using parametric models where the input variable $\bm{y} \in \cU$  typically belongs to a subset of $\bbR^{\bbN}$. These variables may represent parameters simulating complex physical phenomena in {applications arising in} biology, chemistry, economics, engineering {or other fields}. Often, such processes can be represented by a {Hilbert or Banach space-valued} function {$f : \bm{y} \in \cU {\mapsto} f(\bm{y}) \in \cV$} {arising} as a solution of a (system of) parametric {Differential Equations (DEs)}. In this work, and in various applications, the underlying model involves an infinite-dimensional (instead of a high-dimensional) function  to describe the physical process. For instance,  Karhunen-Lo\`{e}ve expansions are used to model random process{es} in many fields{. See, e.g.,} \cite[Sec.~2.1]{le-maitre2010spectral} for further examples.

In most such applications the function of interest is unknown and expensive to evaluate through either numerical simulations or costly physical experiments. Therefore, much effort has focused on the task of constructing \textit{surrogate models}, i.e.,\ approximations $\hat{f}$ that are accurate, but computationally inexpensive to use. While there are several different methodological approaches to this task, in the \textit{nonintrusive} approach one strives to construct $\hat{f}$ from a finite set of {pointwise} \textit{samples} of $f$. Since evaluating $f$ is expensive, it is a crucial and a challenging task to design methods which achieve as good approximation as possible using as few samples as possible.

Various methods have been developed for nonintrusive surrogate model construction. A partial list includes sparse polynomial approximation \cite{beck2012optimal,bieri2010sparse,chkifa2015breaking,cohen2015approximation,cohen2010convergence,hansen2013analytic,hansen2013sparse,chkifa2018polynomial},   Gaussian process regression (or kriging) \cite{smith2013uncertainty,sullivan2015introduction}, radial basis functions \cite{smith2013uncertainty,jung2010recovery}, reduced-basis methods \cite{hesthaven2015certified,quarteroni2015reduced}, and most recently deep neural networks and deep learning \cite{adcock2022nearoptimal,adcock2021deep,adcock2020between,dung2021deep,dung2021computation,
daws2019analysis,opschoor2022exponential,herrman2022constructive,opschoor2022deep,schwab2021deep,li2020better}.  Many of these methods have demonstrated high practical efficacy on parametric DE problems. On the theoretical side, algebraic rates of convergence of the form $\ord{(m/\mathrm{polylog(m)})^{1/2-1/p}}$ have been shown for several such methods, including sparse polynomials and deep learning \cite{adcock2022sparse,chkifa2015breaking,schwab2019deep,adcock2022efficient,adcock2022nearoptimal}, when applied to the class of \textit{Hilbert-valued,} \textit{$(\bm{b},\varepsilon)$-holomorphic} functions \cite{adcock2022sparse,chkifa2015breaking,schwab2019deep,cohen2018shape}. Here $m$ is the number of {pointwise samples}, which are typically Monte Carlo samples, $0 < p < 1$ is a smoothness parameter, $\varepsilon > 0$ and $\bm{b} \in \ell^p(\bbN)$ is a sequence defining the specific class {which determines the {anisotropy} of the functions within the class}. {See \S\ref{S:class_f} for a formal definition of this class.}
 
However, to the best of our knowledge, few works have sought to determine the underlying limits of approximability from samples for such function classes. Note that \cite{bachmayr2016widths} considers Kolmogorov widths for {Hilbert-valued functions $\bm{y} \mapsto u(\bm{y})$ arising as solutions to certain} parametric elliptic PDEs, but does not address the question of finite samples. Such a question is the focus of this paper. Our main results estimate the \textit{$m$-widths} for these function classes. {We first establish lower bounds, which show that no method {-- i.e., no combination of an arbitrary (adaptive) sampling operator and (potentially nonlinear) reconstruction map --} can achieve better rates than $m^{1/2-1/p}$ within this class for $\bm{b} \in \ell^p(\bbN)$ and $0 < p<1$.
These lower bounds are close to the rates achieved by the methods discussed above. Hence, our results {indicate} that these methods are near optimal {(we use the term `indicate' here since {the} bounds for these methods are generally nonuniform with respect to the function, whereas our lower bounds are uniform -- see \S\ref{S:conclusions} for further discussion)}. We also establish upper bounds, which show that the rate $m^{1/2-1/p}$ can be achieved up to a possible log factor by using a suitable random sampling operator and reconstruction map. For both bounds, we study two distinct scenarios, { as previously delineated in} \cite{adcock2022nearoptimal}. First, the \textit{known anisotropy} setting, where $\bm{b}$ is known and therefore can be used by the reconstruction map. Second, the more challenging \textit{unknown anisotropy} setting, where $\bm{b}$ is unknown.

Generally speaking, our work is related to recent advances \cite{binev2022optimal,temlyakov2021optimal} in optimal recovery  \cite{devore1989optimal,micchelli1977survey}. We use concepts and ideas from information-based complexity {\cite{novak1988deterministic,traub1988ibc,novak2010trac,novak2008trac}}, in particular $m$-widths, to understand the aforementioned limits. More specifically, we use lower estimates for the  Gelfand widths, Kolmogorov widths and in general  $m$-widths theory  (see, e.g., \cite{pinkus1968n-widths}).

{The outline of the remainder of this paper is as follows. In  \S\ref{S:prelim} we first introduce a series of preliminaries and notation that will be used later. Next, in \S\ref{s:summary} we introduced the specific function spaces used in this work. In \S\ref{s:summary_2} we  summarize the problems studied  and then present our main results.   These are then proved in \S\ref{S:proofs}. Finally, we conclude in \S\ref{S:conclusions} by listing a number of open problems. We also have two appendices (Appendices \ref{Ap:poly} and \ref{s:widths}), which present several technical results needed in the main arguments.}

\section{Preliminaries}\label{S:prelim}

{We now  introduce some general notation, further preliminaries and setup.}

\subsection{Notation}

{We denote $\bbN$ and $\bbN_0$ as the set of positive and nonnegative integers, respectively.} Let $N \in \mathbb{N}$. We write $[N] = \{1,2,\ldots,N\}$ with the convention $[N] = \bbN$ when $N = \infty$. We denote the real vector space of dimension $N$ by $\mathbb{R}^N$, and the space of real-valued sequences indexed over $\mathbb{N}$ by $\mathbb{R}^{\mathbb{N}}$. In either space, we write $\bm{e}_j=(\delta_{j,k})_{k \in [N]}$ for the standard basis vectors, where $j \in [N]$. 

In either finite or infinite dimensions, we write $\0$ and $\bm{1}$ for the multi-indices consisting of all zeros and all ones, respectively. The inequality $\bm{\mu} \leq  \bnu $ is understood componentwise for any multi-indices $\bm{\mu}  = (\mu_i)_{i \in [{N}]}$ and  $\bnu = (\nu_i)_{i \in [{N}]}$, i.e., $\bm{\mu} \leq \bm{\nu}$ means that $\mu_k \leq  \nu_k$ for all $k \in [{N}]$.
We write $\w \odot \v = (w_iv_i)_{i \in [{N}]}$ {for} the Hadamard product {of} vectors $\bm{w} = (w_i)_{i \in [{N}]}$ and $\bm{v} = (v_i)_{i \in [{N}]}$.  {Let $\bnu =(\nu_i)_{i \in [N]}$ and $\bm{\mu} =(\mu_i)_{i \in [N]}$. Then we also write
\begin{equation*}
 \bnu^{\bm{\mu}} = \prod_{i \in [N]} \nu_i^{\mu_i}, \quad \text{ and } \quad  \bnu ! = \prod_{i \in [N]} (\nu_i !),
 \end{equation*} 
with the convention that $0^0=1$.}

Now let  ${1} \leq p \leq \infty$ and $\w = (w_i)_{i \in \bbN} > \bm{0}$ be a sequence of positive weights. We define the weighted $\ell^p$-space $\ell^p(\bm{w})$ to be the set of all sequences $\z=(z_i)_{i \in \bbN} \in \bbR^{\bbN}$ for which the weighted {(quasi-)} norm $\|{\bm{z}}\|_{p,\w}$ is finite. That is
\begin{equation*}
\|{\bm{z}}\|_{p,\w} : = \left \{ \begin{array}{lc} \left (  \sum_{i \in \bbN } w_i^{-p}{|z_i|^p}\right )^{1/p}< \infty, &  p < \infty ,
\\ 
\sup_{i \in \bbN }\left\lbrace w_i^{-1}{|z_i|} \right\rbrace < \infty, & p = \infty. \end{array} \right . 
\end{equation*}
When $\bm{w} = \bm{1}$ is the vector of ones, we just write $(\ell^p,\nm{\cdot}_{p})$. For  $N \in\bbN$ and $\bm{w} = (w_i)_{i \in [N]}$, we use $\ell^p_N(\bm{w})$ to denote the  finite dimensional space $(\bbR^N,\nm{\cdot}_{p,\w})$ of vectors of length $N$. {When $\bm{w} = \bm{1}$, we just write $(\ell^p_N,\nm{\cdot}_{p})$.}

Next, we write
\begin{equation}\label{eq:defB_weighted}
B^p_N(\bm{w}) = \begin{cases} \left \{ \bm{x} = (x_i)^{N}_{i=1} \in \bbR^N : 
\left (\sum^{N}_{i=1}  {w_i}^{-p} {|x_i|^p}\right )^{1/p} \leq 1 \right \} & p < \infty ,
\\ 
\left \{ \bm{x} = (x_i)^{N}_{i=1} \in \bbR^N :\max_{i=1,\ldots,N} \left \{  w_i^{-1}{|{x}_i|}  \right \} \leq 1 \right \} & p = \infty,
\end{cases}
\end{equation}
for the weighted $\ell^p$-norm {(quasi-norm)} unit ball when $p\geq 1$ {( $0<p<1$)}. When $\bm{w} = \bm{1}$, we simply write $B^p_N$.  

Finally, given a set $S \subseteq [N]$, $\bm{x}_S$ is the vector with  $i$th entry equal to $x_i$ if $i \in S$ and zero otherwise{. We also write $S^c$ for the set complement $[N] \setminus S$ of $S$.}

\subsection{Lebesgue--Bochner spaces}\label{S:function_spaces}
In this paper, we focus on functions of the form  $f:\cU \rightarrow \cV$,
where $\cV$ is a Banach space and $\cU=[-1,1]^{\bbN}$ is the infinite-dimensional symmetric hypercube. We write $\bm{y} = (y_i)_{i \in \bbN}$ for the variable over $\cU$. We construct a probability measure  $\varrho$ as follows. First, we consider a nonnegative Borel probability measure $\rho$ on the interval $[-1,1]$. We then define $\varrho$ as a tensor product of measures $\rho$, and we denote this by
\be{
\label{tensor-product-measure}
\varrho = \rho \times \rho \times \cdots .
}
The existence of  such  measure on $\cU$ is guaranteed by the Kolmogorov  extension theorem (see, e.g. \cite[\S2.4]{tao2011introduction}). A typical example we consider in this paper is case where $\varrho$ is the uniform probability measure, i.e., the tensor product of the univariate measure $\D \rho(y) = \frac12 \D y$. 

Now, given $\varrho$ and $1 \leq q \leq \infty $, we write $L^q_{\varrho}(\cU ; \cV)$ for the Lebesgue-Bochner space of (equivalence classes of) strongly $\varrho$-measurable functions $f: \cU \rightarrow \cV$  with finite norm
\begin{equation}\label{def_Bochnernorm}
\nm{f}_{L^q_{\varrho}(\cU;\cV)} : = 
\begin{cases} 
\left( \int_{\cU} \nm{f( \y)}_{\cV}^q \D \varrho (\y) \right)^{1/q} & 1 \leq q < \infty ,
\\
\mathrm{ess} \sup_{\y \in \cU} \nm{f(\y)}_{\cV}  & q = \infty.
\end{cases}
\end{equation}
For simplicity if $\cV = \bbR$ we simply write $ L^2_{\varrho}(\cU;\bbR)=L^2_{\varrho}(\cU)$.

\section{The space of $(\bm{b},\varepsilon)$-holomorphic functions and anisotropy} \label{s:summary}


We commence this section with a precise definition of the class of functions considered in this paper and a brief discussion of the known and unknown anisotropy cases.

\subsection{$(\bm{b},\varepsilon)$-holomorphic functions}\label{S:class_f}

Let $(\cV,\nm{\cdot}_{\cV})$ be a Banach space and $\cU = [-1,1]^{\bbN}$ be the infinite-dimensional symmetric hypercube of side length $2$. For a parameter $\bm{\rho} = (\rho_j)_{j \in \bbN} \geq \bm{1}$, with the inequality being understood component-wise, let $\cE_{\bm{\rho}} = \cE_{\rho_1} \times \cE_{\rho_2} \times \cdots \subset {\bbC}^{\bbN}$ denote the Bernstein polyellipse, where, for $\rho > 1$, $\cE_{\rho} = \{ \frac12 (z+z^{-1}) : z \in {\bbC},\ 1 \leq | z | \leq \rho \} \subset {\bbC}$ is the classical Bernstein ellipse and, by convention, $\cE_{\rho} = [-1,1]$ when $\rho = 1$. { Note that the major and minor semi-axis lenghts of the ellipse $\cE_{\rho}$ are given by $\frac12 (\rho \pm \rho^{-1})$ and its foci are at $\pm 1$.} {Moreover, classical polynomial approximation theory asserts that any function of one variable that is holomorphic within $\cE_{\rho}$ can be approximated up to an error depending on $\rho^{-n}$ by a polynomial of degree $n$. See, e.g., \cite[Chpt.\ 8]{trefethen2013approximation}.}

Now let $\bm{b} = (b_i)_{i\in\bbN} \in [0,\infty)^{\bbN} $ and $\varepsilon > 0$. A function $f : \cU \rightarrow \cV$ is \textit{$(\bm{b},\varepsilon)$-holomorphic} {\cite{adcock2022sparse,chkifa2015breaking,schwab2019deep,cohen2018shape}} if it is holomorphic in the region
\begin{equation}
\label{def:b-eps-holo}
\cR(\bm{b},\varepsilon) = \bigcup \left\lbrace  {\cE_{\bm{\rho}} } : \bm{\rho} \in [1,\infty)^{\bbN},  \sum^{\infty}_{i=1} \left ( \frac{\rho_i + \rho^{-1}_i}{2} - 1 \right ) b_i \leq \varepsilon  \right\rbrace \subset \bbC^{\bbN}.
\end{equation}
In this paper, without loss of generality we consider $\varepsilon = 1$,  a property which can be guaranteed by rescaling $\bm{b}$. For convenience, we write $\cR(\bm{b})$ instead of $\cR(\bm{b},1)$ and define
\bes{
\cH(\bm{b}) = \left \{ f : \cU \rightarrow \cV\text{ $(\bm{b},1)$-holomorphic} : \nm{f}_{L^{\infty}(\cR(\bm{b});\cV)} {= :\esssup_{\bm{z} \in \cR(\bm{b})} \nm{f({\bm{z}})}_{\cV} } \leq 1 \right \}.
}
The {class of $(\bm{b},\varepsilon)$-holomorphic functions} was first studied in the context of parametric DEs. As shown in the series of works \cite{adcock2022sparse,cohen2015approximation,chkifa2015breaking,aylwin2020domain,schwab2011sparse
,castrillon--candas2016analytic,cohen2010convergence,gunzburger2014stochastic,hansen2013analytic,hansen2013sparse,cohen2018shape,
hoang2012regularity,schwab2011sparse,traonmilin2018stable}, the solution maps of many common parametric DEs are $(\bm{b},\varepsilon)$-holomorphic functions of their parameters. See \cite[Ch.~4]{adcock2022sparse} and \cite{cohen2015approximation} for overviews. In tandem with the effort, there has also been a focus on the approximation theory of the class $\cH(\bm{b})$. This theory is now well developed.  A standard setting is to consider the case where $\bm{b} \in \ell^p(\bbN)$ for some $0 < p <1$. A signature result then shows that the best $s$-term polynomial approximation to any $f \in \cH(\bm{b})$ converges with rate $s^{1/2-1/p}$ in the Lebesgue--Bochner norm {$\nm{\cdot}_{L^2_\varrho(\cU;\cV)}$ defined in \eqref{def_Bochnernorm} above}. See \cite{chkifa2015breaking,cohen2010convergence} {in the case where} $\varrho$ {is} the uniform measure and \cite{rauhut2017compressive} in the case of the Chebyshev measure. {See also \cite[Ch.\ 3]{adcock2022sparse} and Corollary \ref{cor:known}.}

{The best $s$-term approximation is a theoretical benchmark. Unfortunately, it is often of limited practicality for computations, since} the largest $s$ polynomial coefficients can occur at arbitrary indices.
For this reason, it is also common to {make the additional assumption} that $\bm{b} \in \ell^p_{\mathsf{M}}(\bbN)$, where $\ell^p_{\mathsf{M}}(\bbN)$ is the \textit{monotone} $\ell^p$ space. This is the space of sequences whose minimal monotone majorant is $\ell^p$ summable, where, given  a sequence  $\bm{z}=(z_i)_{i \in \bbN} \in \bbR^{\bbN}$,   its \textit{minimal monotone majorant} is defined as 
\begin{equation}
\label{min-mon-maj}
\tilde{\bm{z}} = (\tilde{z}_i)_{i \in \bbN},\quad \text{where }
\tilde{z}_i = \sup_{j \geq i} | z_{j}|,\ \forall i \in \bbN.
\end{equation}
Given $0 < p < \infty$, the  \textit{monotone $\ell^p$-space} $\ell^p_{\mathsf{M}}(\bbN)$ is {then} defined by
\begin{equation*}
\ell^p_{\mathsf{M}}(\bbN) = \{ \bm{z} \in \ell^{\infty}(\bbN) : \nmu{\bm{z}}_{p,\mathsf{M}} : = \nmu{\tilde{\bm{z}}}_{p} < \infty  \}.
\end{equation*}
{When $\bm{b} \in \ell^p_{\mathsf{M}}(\bbN)$} one can attain the same rate {$s^{1/2-1/p}$}, but with a more structured polynomial approximation {that is more amenable to computations, since its} nonzero indices belong to a so-called \textit{anchored set}. {See \S\ref{Ap:anch} and, in particular, Corollary \ref{cor:anchored_sigma}.} We consider both settings, $\bm{b} \in \ell^p(\bbN)$ and $\bm{b} \in \ell^p_{\mathsf{M}}(\bbN)$, in this paper.

\subsection{Known versus unknown anisotropy}\label{S:known-unk}

We refer to $\bm{b}$ as the \textit{anisotropy} parameter of a function $f \in \cH(\bm{b})$. When $b_i$ is large, the condition
\bes{
\sum^{\infty}_{i=1} \left ( \frac{\rho_i + \rho^{-1}_i}{2} - 1 \right ) b_i \leq 1,
}
appearing in the definition \R{def:b-eps-holo} holds only for smaller values of $\rho_i$, meaning that $f$ is less smooth with respect to the variable $y_i$. Conversely, when $b_i$ is small (or zero), $f$ is smoother with respect to $y_i$. For further details we refer to \cite{adcock2022nearoptimal,adcock2021deep,adcock2022sparse}. 

An important consideration in this paper is whether $\bm{b}$ is known or unknown. In some applications, one may have knowledge of $\bm{b}$. However, as mentioned in \cite{adcock2022nearoptimal}, in practical UQ settings, where $f$ is considered a black box, the behaviour of the target function with respect to its variables (and therefore $\bm{b}$) is a priori unknown.
 In the \textit{known anisotropy} setting, a reconstruction procedure can use $\bm{b}$ to achieve a good approximation uniformly over the class $\cH(\bm{b})$. However, in the \textit{unknown anisotropy} setting, the reconstruction procedure has no access to $\bm{b}$. In this case, motivated by the best $s$-term approximation theory described above, we fix a $0 < p < 1$ and consider the classes
\begin{equation*}
\cH(p) =  \bigcup \left \{ \cH(\bm{b}) : \bm{b} \in \ell^p(\bbN),\ \bm{b} \in [0,\infty)^{\bbN},\   \nm{\bm{b}}_p \leq 1 \right \} 
\end{equation*}
and
\begin{equation*}
\cH(p,{\mathsf{M}})  = \bigcup  \left \{ \cH(\bm{b}) : \bm{b} \in \ell^p_{\mathsf{M}}(\bbN),\ \bm{b} \in [0,\infty)^{\bbN},\  \nm{\bm{b}}_{p,\mathsf{M}} \leq 1 \right \}  .
\end{equation*}

\section{Summary and main results}\label{s:summary_2}

This section introduces the adaptive sampling operators, the definitions of (adaptive) $m$-widths, the main results in this work and discusses their implications.

\subsection{Sampling operators}\label{S:sampling_scalar}


We now formalize what we mean by an adaptive sampling operator. {Note that this is commonly referred to as \textit{adaptive information} in the field of information-based complexity \cite[Sec.~4.1.1]{novak2008trac}. }
We commence with the scalar-valued case $\cV = \bbR$.

\defn{
[Adaptive sampling operator; scalar-valued case]
\label{def:adaptive-sampling-scalar}
 Consider a normed vector space $(\cY,\nm{\cdot}_{\cY})$. A scalar-valued    \textit{adaptive sampling operator} is a map of the form
\bes{
\cL : \cY  \rightarrow \bbR^m,\ \quad \cL(f) = \begin{bmatrix} L_1(f) \\ L_2(f ; L_1(f)) \\ \vdots \\ L_m(f ; L_{1}(f),\ldots, L_{m-1}(f)) \end{bmatrix},
}
where  $L_1 : \cY \rightarrow \bbR$ is a bounded  linear functional and, for $i = 2,\ldots,m$, $L_i :\cY \times \bbR^{i-1} \rightarrow \bbR$ is bounded and linear in its first component.
}
Notice that any linear map $\cL : \cY \rightarrow \bbR^m$ is an adaptive sampling operator. The rationale for considering adaptive sampling operators is to cover approximation methods where each subsequent sample is chosen adaptively in terms of the previous measurements. 

Note also that this definition allows for arbitrary adaptive sampling operators. An important special case is that of (adaptive) pointwise samples (so-called \textit{standard information} \cite[Sec.~4.1.1]{novak2008trac}). Let $\cY = C(\cU)$. Then this is defined as
\be{
\label{scalar-valued-pointwise-adaptive}
\cL(f) = (f(\bm{y}_i))^{m}_{i=1} \in \bbR^m,\quad \forall f \in C(\cU),
}
where $\bm{y}_i$ is the $i$th sample point, which is potentially chosen adaptively based on the previous measurements $f(\bm{y}_1),\ldots,f(\bm{y}_{i-1})$.

Next, we consider the Banach-valued case. In the following definition, for any $w \in \cV$ and $\bm{v} = (v_i)^{m}_{i=1} \in \bbR^m$, we write $w \bm{v}$ for the vector $(w v_i )^{m}_{i=1} \in \cV^m$. Note that in this definition, we consider   Lebesgue--Bochner spaces only, as opposed to arbitrary Banach spaces. 

\defn{
[Adaptive sampling operator; Banach-valued case]
\label{def:L-B}
Consider a vector space  $\cY \subseteq L^2_\varrho(\cU;\cV)$ with norm $\|{\cdot}\|_{\cY}$ {and an operator}
\bes{
\cL : \cY  \rightarrow \cV^m,\ \quad {\cL(f)} = \begin{bmatrix} L_1(f) \\ L_2(f ; L_1(f)) \\ \vdots \\ L_m(f ; L_{1}(f),\ldots, L_{m-1}(f)) \end{bmatrix},
} 
where $L_1 : \cY   \rightarrow \cV$ is a bounded linear operator and, for $i = 2,\ldots,m$, $L_i : \cY \times \cV^{i-1} \rightarrow \cV$ is a bounded linear operator in its first component. {Then $\cL$ is a \textit{Banach-valued adaptive sampling operator} if the following condition holds.} There exist $v,w \in \cV \backslash \{0\}$,  a normed vector space $\widetilde{\cY} \subseteq L^2_{\varrho}(\cU)$ and a scalar-valued adaptive sampling operator $\widetilde{\cL} : \widetilde{\cY} \rightarrow \bbR^m$ (see Def.~\ref{def:adaptive-sampling-scalar})
such that if $v g \in \cY$ for some $g \in L^2_{\varrho}(\cU)$ then $g \in \widetilde{\cY}$ and $\cL(v g) = w \widetilde{\cL}(g)$.
}
Note that the condition imposed in Def.~\ref{def:L-B} is not a strong one. For example, it trivially holds in the important case of (adaptive) pointwise sampling. Here, we consider $\cY = C(\cU ; \cV)$ and
\bes{
\cL(f ) = (f(\bm{y}_{i}))^{m}_{i=1} \in \cV^m,\quad \forall f \in \cY,
}
where $\bm{y}_{i} \in \cU$ is the $i$th sample point,  which is potentially chosen adaptively based on the previous measurements $f(\bm{y}_1),\ldots,f(\bm{y}_{i-1})$. In this case, we clearly have
\bes{
\cL(v g) = v \widetilde{\cL}(g),\quad \forall g \in \widetilde{\cY} : = C(\cU),\ v \in \cV,
}
where $\widetilde{\cL}$ is the scalar-valued (adaptive) pointwise sampling operator \R{scalar-valued-pointwise-adaptive}. 

{The condition imposed in Def.~\ref{def:L-B} is used to establish our lower bounds -- in particular, the reduction to a discrete problem in Lemma \ref{l:discrete-reduction-known} and elsewhere}. It is an open problem whether these bounds hold in the Banach-valued case without this assumption.

\subsection{Adaptive $m$-widths} \label{S:sampling}

Let $(\cV,\nm{\cdot}_{\cV})$  be a  Banach space. In this paper, we study the \textit{(adaptive)   $m$-width} (which is related to the \textit{information complexity} \cite[\S4.1.4]{novak2008trac}) 
\be{
\label{Theta_m-def}
\Theta_m(\cK ; \cY,\cX) =  \inf \left \{ \sup_{f \in \cK} \nmu{f - \cT ( \cL(f)) }_{\cX} : \cL : \cY \rightarrow \cV^m\text{ adaptive},\ \cT : \cV^m \rightarrow \cX   \right \},
}
where $\cY$ is a normed vector subspace of the Lebesgue--Bochner space  $\cX = L^2_{\varrho}(\cU ; \cV)$  with  tensor-product  probability measure $\varrho$ and $\cK \subseteq \cY$. 
Here $\cL$ is {an adaptive sampling operator, as in Defs.\ \ref{def:adaptive-sampling-scalar} and \ref{def:L-B}.} 

We now elaborate \R{Theta_m-def}. Given $f \in \cK \subseteq \cY$, the adaptive sampling operator yields $m$ measurements belonging to the {underlying} Banach space $\cV$. The term $\cT$ is an arbitrary \textit{reconstruction map}, which takes a vector of $m$ $\cV$-valued measurements  and produces an approximation in $\cX$. Thus, $\Theta_m(\cK ; \cY,\cX)$ pertains to the optimal sampling and reconstruction maps for minimizing the worse-case recovery error over the set $\cK$. Note that the choice of $\cY$ determines the type of allowed sampling operators. If $\cY = C(\cU ; \cV)$ {is the space of continuous functions from $\cU$ to $\cV$  with respect to the uniform norm},
 for example, then this includes (adaptive) pointwise sampling, whereas if $\cY = \cX$ then it does not. {Note, however, that $\cY$ plays a quite minor role in our analysis: our lower bounds (see \S \ref{ss:lower-bounds}) hold for arbitrary $\cY$, while our upper bounds (see \S \ref{ss:upper-bounds}) require the (mild) assumption that $\cY$ is compactly contained in $\cX$. {We provide some additional discussion on $\cY$ in Remark~\ref{rm:ChoiceY} below.}} 

Observe that \R{Theta_m-def} generalizes standard definitions \cite{binev2022optimal}, where the measurements $\cL(f) \in \bbR^m$ are scalar-valued. We introduce this extension to allow for Banach-valued measurements, as this is relevant for sampling-based approximation methods {for} parametric PDEs.

 Our main contributions in this work are lower and upper bounds for \R{Theta_m-def} when  $\cK$ is taken as $\cK = \cH(\bm{b})$ in the case of known anisotropy and $\cK = \cH(p)$ or $\cK = {\cH(p,\mathsf{M})}$  in the case of unknown anisotropy. 
 For convenience, we define
 \be{
 \label{theta-upsilon-known-aniso}
\theta_m(\bm{b}) = \Theta_m(\cH(\bm{b}) ; \cY, L^2_{\varrho}(\cU;\cV)), 
 }
 in the case of known anisotropy. Observe that this is the $m$-width where the optimal reconstruction map can (and generally will) depend on the anisotropy parameter $\bm{b}$. As discussed above, our interest lies in the case where $\bm{b} \in \ell^p(\bbN)$ or $\bm{b} \in \ell^p_{\mathsf{M}}(\bbN)$ for some $0 < p <1$. Thus, we also define
 \begin{equation}\label{eq:def_theta} 
\begin{split}
\overline{\theta_m}(p) &= \sup \left \{ \theta_m(\bm{b}) : \bm{b} \in \ell^p(\bbN),\ \bm{b} \in [0,\infty)^{\bbN},\  \nm{\bm{b}}_p \leq 1 \right \},
\\
\overline{\theta_m}(p,{\mathsf{M}}) &= \sup \left \{ \theta_m(\bm{b}) : \bm{b} \in \ell^p_{\mathsf{M}}(\bbN),\ \bm{b} \in [0,\infty)^{\bbN},\  \nm{\bm{b}}_{p,\mathsf{M}} \leq 1 \right \}.
\end{split}
\end{equation} 
Finally, in the unknown anisotropy setting, we define 
 \begin{equation}\label{theta-upsilon-unknown-aniso}
\theta_m(p) = \Theta_m(\cH(p) ; \cY, L^2_{\varrho}(\cU;\cV)  ), \quad 
\theta_m(p,{\mathsf{M}}) = \Theta_m(\cH(p,{\mathsf{M}}) ;\cY, L^2_{\varrho}(\cU;\cV)  ).
 \end{equation}
{
Notice that $\theta_m(p)$ is not equivalent to $\overline{\theta_m}(p)$, and likewise for $\theta_m(p,\mathsf{M})$ and $\overline{\theta_m}(p,\mathsf{M})$. The former pertains to reconstruction maps that are not permitted to depend on $\bm{b}$, whereas the latter pertains to reconstruction maps that can. In particular, we have}
\begin{equation*}
{ \theta_m(p) \geq \overline{\theta_m}(p) \geq \overline{\theta_m}(p,\mathsf{M}) \text{ and } \theta_m(p,\mathsf{M}) \geq \overline{\theta_m}(p,\mathsf{M}).}
 \end{equation*}

 \begin{remark}\label{rm:ChoiceY}
 {Note that one cannot set $\cY=\cK$ in \R{theta-upsilon-known-aniso} or \R{theta-upsilon-unknown-aniso}, since $\cK$ is not a linear subspace of $L^2_\varrho(\cU;\cV)$. In the known anisotropy case, a natural choice would be to set $\cY = \cG(\bm{b})$, where $\cG(\bm{b})$ is vector space of $(\bm{b},1)$-holomorphic functions equipped with the $L^{\infty}(\cR(\bm{b});\cV)$-norm. In this case, $\cH(\bm{b})$ is the unit ball of $\cG(\bm{b})$. However, there is no similar such choice in the cases of $\cH(p)$ (or $\cH(p,\mathsf{M})$). It follows immediately from the definition that $\cH(p)$ is a union over $\bm{b} \in \ell^p(\bbN)$, $\nm{\bm{b}}_p \leq 1$ of unit balls of the subspaces $\cG(\bm{b})$, each equipped with a different norm. Therefore, $\cH(p)$ is contained in a union of subspaces
 \begin{equation*}
\cH(p) = \bigcup_{\|\bm{b}\|_p \leq 1} \cH(\bm{b}) \subset  \bigcup_{\|\bm{b}\|_p \leq 1} \cG(\bm{b}).
\end{equation*}
However, it is possible to show that the right-hand side is not itself a subspace. Therefore, there is no intrinsic choice for $\cY$ in the unknown anisotropy setting.}
\end{remark}

\subsection{Main results}\label{S:main_res}

\subsubsection{Lower bounds} \label{ss:lower-bounds}
We first consider lower bounds. {Note that these bounds hold for any choice of the normed vector space $\cY$ appearing in \R{Theta_m-def} {that contains $\cK = \cH(\bm{b})$ (known anisotropy case) or $\cK = \cH(p),\cH(p,\mathsf{M})$ (unknown anisotropy case)}.} 
We first consider the known anisotropy case. To state the corresponding result, we now recall the definition of the $\ell^2$-norm \textit{best $s$-term approximation error} of a sequence $\bm{c} \in \ell^2(\bbN)$. This is given by
\bes{
\sigma_s(\bm{c})_2 = \min \{ \nm{\bm{c} - \bm{z}}_2 : \bm{z} \in \ell^2(\bbN),\ | \mathrm{supp}(\bm{z}) | \leq s \},
}
where $\supp(\bm{z}) = \{ i : z_i \neq 0\}$ for $\bm{z} = (z_i)_{i \in \bbN} \in \bbR^{\bbN}$.

\begin{theorem}
[Known anisotropy; lower bounds and the rate $m^{1/2-1/p}$ ]
\label{thm:known-lower} 
Let  $m\geq1$ and $\varrho$ be a tensor-product probability  measure on $\cU$. Then the following hold. 
\begin{itemize}
\item[(a)]  For every   $\bm{b} \in [0,\infty)^{\bbN}$ with $\bm{b} \in \ell^1(\bbN)$, the $m$-width \R{theta-upsilon-known-aniso} satisfies
\bes{
\theta_m(\bm{b}) \geq c \cdot \sigma_m(\bm{b})_2,
}
where $c > 0$ depends on the measure $\varrho$  and $\|\b\|_1$ only.
\item[(b)] For every $0 < p <1$, the $m$-widths \R{eq:def_theta} satisfy
\begin{equation}\label{eq:part:b}
 \overline{\theta_m}(p)  
\geq \overline{\theta_m}(p,{\mathsf{M}}) \geq c \cdot 2^{-1/p} \cdot m^{1/2-1/p},
\end{equation}
 where $c > 0$ depends on the measure $\varrho$  only. 
\item[(c)] Let  $(g(n))_{n \in \bbN}$ be a positive {nondecreasing} sequence such that $$\sum_{n \in \bbN} (n g(n))^{-1} < \infty.$$  Then, for every $0 < p < 1$, there exists a ${\bm{b}} \in \ell^{p}_{\mathsf{M}}(\bbN)$, ${\bm{b}} \in [0,\infty)^{\bbN}$  
 such that 
\begin{equation}\label{eq:bound_g}
 \overline{\theta_m}(p) \geq \overline{\theta_m}(p,{\mathsf{M}}) \geq \theta_m(\bm{{b}})  \geq c'\cdot g(2m)^{-1/p} \cdot m^{1/2-1/p},
\end{equation}
where {$c'= c \cdot 2^{-1/p} \cdot\left(\sum_{n \in \bbN}(ng(n))^{-1} \right)^{-1/p}$ and  $c$ is the constant in (b).}
\end{itemize}
\end{theorem}

\pbk
Part (a)  of this result provides a lower bound for the $m$-width $\theta_m(\bm{b})$ in terms of the best $m$-term approximation error $\sigma_m(\bm{b})_2$. The inequality often referred to as Stechkin's inequality (see, e.g., the historical note \cite[Rem. 3.4]{adcock2022sparse} on the origins and naming of this inequality and \cite[Lem.\ 3.5]{adcock2022sparse}) shows that
\bes{
\sigma_m(\bm{b})_2 \leq \nm{\bm{b}}_p (m+1)^{1/2-1/p} ,
}
whenever $\bm{b} \in \ell^p(\bbN)$. {} Hence, the main contribution of part (b) is to show that the rate $m^{1/2-1/p}$ is, in effect, sharp when considering all possible (unit-norm)  $\bm{b} \in \ell^p(\bbN)$ or $\bm{b} \in \ell^p_{\mathsf{M}}(\bbN)$. However, it is notable that  the $\bm{b}$ that achieves this bound depends on $m$ and has equal entries, i.e., it corresponds to a class of the functions $\cH(\bm{b})$ that is completely isotropic (see the proof of Theorem \ref{thm:known-lower} for details). In part (c),  we prove that a nearly sharp lower bound of the form $m^{1/2-1/p}$ can be obtained for a fixed $\bm{b}$, which is independent of $m$ and determined only by a {nondecreasing} function $g$.  Specifically, the function $g$ can be  chosen {to grow} very slowly, such as $g(n) = \log^2(n+1)$ or even $g(n) = \log(n+1) (\log(\log(n+1)))^2$.  For further examples, we refer to  \cite[Ch.~3]{rudin1964principles}.  

\begin{theorem}
[Unknown anisotropy; lower bounds]
\label{thm:unknown-lower-1}
Let $m\geq1$, $\varrho$ be a tensor-product probability  measure  on $\cU$  and $0 < p <1$. Then the following hold.
\begin{itemize}
\item[(a)] The $m$-width $\theta_m(p)$ in \R{theta-upsilon-unknown-aniso} satisfies
\ea{
\label{theta-mp-lower}
\theta_m(p) & \geq c \cdot 
2 ^{1/2-2/p},
}
where $c> 0$ depends on the measure $\varrho$ only.
\item[(b)] The $m$-width $\theta_m(p,{\mathsf{M}})$ in \R{theta-upsilon-unknown-aniso} satisfies
\begin{equation}\label{eq:part:b_unknown}
 \theta_m(p,{\mathsf{M}})  
\geq \overline{\theta_m}(p,{\mathsf{M}}) \geq  c \cdot  2^{-1/p} \cdot m^{1/2-1/p},
\end{equation} 
where $c> 0$ depends on the measure $\varrho$ only. 
\end{itemize}
\end{theorem}

Part (a) of this theorem{, more precisely the factor $2^{1/2-2/p}$,} shows that approximation from finite samples is, in fact, impossible in the space $\cH(p)$, since the $m$-width does not decay as $m \rightarrow \infty$. This is perhaps unsurprising. Functions in $\cH(p)$ are anisotropic, but their (infinitely-many) variables can be ordered in terms of importance in arbitrary {and infinitely-many different} ways. It seems implausible that one could approximate such functions from a finite set of data. This result confirms this intuition.

However, part (b) of this theorem reveals that the situation changes completely when we restrict to the monotone space $\cH(p,{\mathsf{M}})$. Here the lower bound for the $m$-width  is once more $m^{1/2-1/p}$. 

\subsubsection{Upper bounds}\label{ss:upper-bounds}

We now present a series of upper bounds for the various $m$-widths. {In the following results, we make the additional (mild) assumption that the normed vector space $\cY$ appearing in \R{Theta_m-def} is compactly contained in $L^2_\varrho(\cU;\cV)$, i.e.,  $\cY \hookrightarrow L^2_\varrho(\cU;\cV)$.}

\begin{theorem}
[Known anisotropy; upper  {bounds}]
\label{thm:unknown-upper-1}
Let  $m\geq1$, $\varrho$ be the uniform probability measure  on $\cU$ and $0 < p <1$. Then the following hold.
\begin{itemize}
\item[(a)]   {The} $m$-width \eqref{theta-upsilon-known-aniso} satisfies
\begin{equation}
\theta_m(\bm{b}) \leq c \cdot m^{1/2-1/p}, \quad \forall \bm{b} \in \ell^p(\bbN),\ \bm{b} \in [0,\infty)^{\bbN},
\end{equation}
where $c>0$ depends on $\bm{b}$ and $p$ only. Moreover, this bound is attained by a {bounded linear (nonadaptive) sampling map $\cL$ and a bounded linear reconstruction operator $\cT$.} 


\item[(b)] In addition, for any $q \in (p,1)$ 
 the $m$-width \eqref{eq:def_theta} satisfies 
\begin{equation}
\label{theta-mpM-upper}
\overline{\theta_m}(p,{\mathsf{M}}) \leq c \cdot m^{1/2-1/q},
\end{equation}
where $c> 0$ depends on $p$ and $q$ only.
\end{itemize}

\end{theorem}
Part (a) shows that in the case of known anisotropy for a fixed $\bm{b}$ we can achieve a rate   $m^{1/2-1/p}$ as an upper bound with a constant depending on  $\bm{b} \in \ell^p(\bbN)$. In contrast, part (b) provides a uniform bound over all $\bm{b}$ belonging to the unit ball of the monotone $\ell^p$ space $\ell^p_{\mathsf{M}}(\bbN)$. The key difference is the algebraic rate, which can be arbitrarily close to $m^{1/2-1/p}$, but not equal to it.  {Currently, we do not have an explicit expression for how the constant $c$ in part (b) depends on $p$ and $q$ and, in particular, how it behaves as $q \rightarrow p^{+}$, besides the knowledge that it must blow up. In particular, the question of whether \R{theta-mpM-upper} may in fact hold} with $q = p$ is an open problem and, if true, its proof would require a different technique. {See Remark \ref{rem:constants-blowup} for some additional discussion.}

Notice also that part (b) only considers the monotone space $\ell^p_{\mathsf{M}}(\bbN)$. We suspect that achieving a uniform bound for all $\bm{b}$ belonging to the unit ball of the standard $\ell^p$ space $\ell^p(\bbN)$ may not be possible. However, this is also an open problem.

{
Both Theorem \ref{thm:unknown-upper-1} and the next result are proved using Legendre polynomials. In Theorem \ref{thm:unknown-upper-1}, the sampling operator $\cL$ computes $m$ Legendre coefficients of $f$ from a suitable index set (depending on $\bm{b}$ only). The reconstruction map then simply forms the corresponding Legendre polynomial expansion. {In particular, the reconstruction map is linear.} This approach is possible in the setting of known anisotropy, since we can choose the index set in terms of $\bm{b}$.
}

{
In the unknown anisotropy setting, we resort to a different approach based on compressed sensing \cite{foucart2013mathematical,adcock2022sparse}. The idea is to view the polynomial coefficients as an (approximately) sparse vector, whose significant entries are unknown (since $\bm{b}$ is unknown). An optimal, stable way in compressed sensing to `sense' a sparse vector is to use an $m \times N$ Gaussian random matrix $\bm{A}$.  We therefore choose the sampling operator $\cL$ so that the recovery problem for the coefficients reduces to the problem of recovering a sparse vector from random Gaussian measurements. Specifically, $\cL$ computes $m$ random weighted sums of the Legendre coefficients of $f$ where the weights are i.i.d.\ normal random variables (a similar idea was also used in a different context in \cite{adcock2021do}).  Finally, we formulate the (nonlinear) reconstruction map in terms of a convex $\ell^1$-minimization problem.
}

{
In order to use this approach, we first need to restrict the infinite vector of polynomial coefficients to a finite vector of length $N$. This needs to be done carefully so that no large coefficients are excluded from the resulting finite vector. This is where we use the additional assumption $\bm{b} \in \ell^p_{\mathsf{M}}(\bbN)$ and the properties of anchored sets. Concretely, we construct a finite vector by including all polynomial coefficients within a certain \textit{hyperbolic cross} index set, which has the property that it contains all anchored sets of a given size (depending on $m$). See \S \ref{S:proof2.4} for further details. Crucially, this gives a recipe for performing the truncation in which the length $N = N(m)$ of the finite vector is given in terms of $m$ only.
}

{
For notational convenience, in the following result we introduce a normal random vector $\bm{r} \sim \cN(0,I_{\ell})$, $\ell= m N$, which contains the entries of the aforementioned Gaussian random matrix $\bm{A}$.
}

\begin{theorem}
[Unknown anisotropy; upper bounds]
\label{thm:unknown-upper-2}
Let $\cV$ be a Hilbert space,  $m \geq 3$ and $\varrho$ be the uniform probability measure  on $\cU=[-1,1]^{\bbN}$.
Then  there exists an $\ell=\ell(m)$, a (nonlinear) reconstruction map  $\cT_{\bm{r}} : \cV^m \rightarrow L^2_{\varrho}(\cU ; \cV)$ and a (nonadaptive)  bounded linear sampling operator  $\cL_{\bm{r}} : \cY  \rightarrow \cV^m$  depending on a random vector $\bm{r} \sim \cN(0,I_\ell)$, where $I_\ell$ is the $\ell \times \ell$ identity matrix, such that  the following holds.
\begin{itemize}

\item[(a)]   We have
\begin{equation*}
\bbE_{\bm{r} \sim \cN(0,I_\ell) } \sup_{f \in \cH(\bm{b})} \nmu{f - \cT_{\bm{r}}(\cL_{\bm{r}}(f))}_{L^2_{\varrho}(\cU ; \cV)} \leq c \cdot \left( \dfrac{m}{\log^2(m)}\right)^{1/2-1/p},
\end{equation*}
for all $\bm{b} \in \ell^p_{\mathsf{M}}(\bbN)$, $\bm{b} \in [0,\infty)^{\bbN} $, and $0 < p < 1$, where $c > 0$ depends on $\bm{b}$ and $p$ only. 

\item[(b)] For {$0<p<q<1$}, the $m$-width \R{theta-upsilon-unknown-aniso} satisfies
\begin{equation*}
\theta_m(p,{\mathsf{M}}) \leq \bbE_{\bm{r} \sim \cN(0,I_\ell) }  \sup_{f \in \cH(p,\mathsf{M})} \nmu{f - \cT_{\bm{r}}(\cL_{\bm{r}}(f))}_{L^2_{\varrho}(\cU ; \cV)} \leq  c \cdot {m}^{1/2-1/q},
\end{equation*}
where $c > 0$ depends on $p$ and $q$ only. 
\end{itemize}
\end{theorem}

Part (b) of this theorem shows that the lower bounds in Theorem \ref{thm:unknown-lower-1}(b) can be nearly achieved by a (nonadaptive) random sampling operator and reconstruction map. Here, ``nearly'' means  with a rate that can be arbitrarily close to $1/2-1/p$, but not equal to it {. As in the previous result, we also have no explicit expression for the constant $c$ in part (b). This is for much the same reasons as those discussed in Remark \ref{rem:constants-blowup}}. By contrast, Theorem \ref{thm:unknown-upper-2}(a) shows that the algebraic rate $1/2-1/p$ can be achieved, but with a constant depending on $\bm{b}$. Crucially, however, the sampling operator and reconstruction map are both independent of $\bm{b}$ and $p$ {(and $q$)}. In other words, part (a) is also a result about the unknown anisotropy setting, even though the supremum is taken over $\cH(\bm{b})$.

{Theorems \ref{thm:unknown-upper-1} and \ref{thm:unknown-upper-2} consider only the uniform probability measure on $\cU$. We anticipate they hold for more general tensor-product probability measures, such as Jacobi measures. We leave this for future work.} We also note in passing that Theorem \ref{thm:unknown-upper-2} assumes that $\cV$ is a Hilbert space. Whether these rates hold in the Banach-valued case remains an open problem.

\section{Proof of main theorems}\label{S:proofs}

We now prove the main results, Theorems \ref{thm:known-lower}--\ref{thm:unknown-upper-2}.

\subsection{Widths and standard results on widths}\label{S:apxA}

We commence with a brief overview of Gelfand widths and their properties, since these will be crucial in the proofs of the lower bounds. See \cite{pinkus1968n-widths} or \cite[Ch.~10]{foucart2013mathematical} for more details. Let $\cK$ be a subset of a  normed space $(\cX,\nm{\cdot}_{\cX})$. Then its \textit{Gelfand $m$-width} is
\begin{equation}\label{def:mwidths}
d^m(\cK,\cX) = \inf \left \{ \sup_{x \in \cK \cap L^m} \nm{x}_{\cX},\ \text{$L^m$ a subspace of $\cX$ with $\mathrm{codim}(L^m) \leq m$} \right \}.
\end{equation}
An equivalent representation is
\bes{
d^m(\cK,\cX) = \inf \left \{ \sup_{x \in \cK \cap \mathrm{Ker}(A)} \nm{x}_{\cX},\ A : \cX \rightarrow \bbR^m\text{ linear} \right \}.
}
The Gelfand width is related to the following quantity:
\begin{equation}\label{eq:def-E_ada}
E^m_{\mathsf{ada}}(\cK,\cX) = \inf \left \{ \sup_{x \in \cK} \nm{x - \Delta(\Gamma ( x)) }_{\cX},\ \Gamma : \cX \rightarrow \bbR^m\text{ adaptive},\ \Delta : \bbR^m \rightarrow \cX \right \},
\end{equation}
where $\Gamma $ is an adaptive sampling operator  as in Def.\ \ref{def:adaptive-sampling-scalar}.
This is referred to as the \textit{adaptive compressive $m$-width} of a subset $\cK$ of a normed space $(\cX,\nm{\cdot}_{\cX})$ in \cite[Ch.~10]{foucart2013mathematical}. Note that,  if $\cX = L^2_{\varrho}(\cU ; \bbR)$ {and $\cV = \bbR$}, then  $E^m_{\mathsf{ada}}(\cK,\cX)$  {coincides with} $\Theta_m(\cK;\cX,\cX)$ in \eqref{Theta_m-def}.

The following result is standard and can be found in \cite[Thm.\ 10.4]{foucart2013mathematical}.
\begin{theorem}\label{thm:widths_prop}
Let  {$\cK \subseteq \cX$,} where $(\cX,\nm{\cdot}_{\cX})$ is a normed space. If $-\cK = \cK$ then $$d^m(\cK,\cX) \leq E^m_{\mathsf{ada}}(\cK,\cX).$$
\end{theorem} 
Finally, we also define the \textit{Kolmogorov $m$-width} of a subset $\cK$ of a normed space $(\cX,\nm{\cdot}_{\cX})$  as
\bes{
d_m(\cK,\cX) = \inf \left \{ \sup_{x \in \cK} \inf_{z \in \cX_m} \nm{x - z}_{\cX},\text{ $\cX_m$ a subspace of $\cX$ with $\dim(\cX_m) \leq m$} \right \}.
}

\subsection{Proof of Theorem \ref{thm:known-lower}}

{The proofs of our lower bounds} rely on the reduction of the {continuous} problem to a {discrete one}. First, we prove a holomorphy result of order-one polynomials. {This allows us to lower bound the $m$-width $\theta_m(\bm{b})$ in terms of a certain discrete problem.} 
Having done this, we then use certain results on the Gelfand and Kolmogorov widths of weighted $\ell^p$ balls to get the {desired} bound.

We start by making the following observation. Let $\rho $  be  a probability  measure as in \S\ref{S:function_spaces} and recall that $\varrho$ is the tensor-product probability measure defined in \eqref{tensor-product-measure}. 
{Since $\rho$ is a nonnegative Borel probability measure on the interval $[-1,1]$} its moments exist \cite[Sec. 6.1]{Walter2001wavelets}. Now let
\begin{equation}\label{def:alphabeta}
\tau = \int^{1}_{-1} y \D \rho(y), \quad \text{ and } \quad  \sigma = \sqrt{\int^{1}_{-1} (y-\tau)^2 \D \rho(y)},
\end{equation}
{be the first and second moment of $\rho$ and notice that $\tau, \sigma < \infty$.}
Then {the functions}
\bes{
\psi_i(\bm{y}) = \frac{y_i - \tau}{\sigma},\quad \bm{y} = (y_j)_{j \in \bbN} \in \cU,\ i \in \bbN,
}
{form an orthonormal set $\{\psi_i\}_{i \in \bbN} \subset L^2_\varrho(\cU)$ (but not a basis).} 
Furthermore, if $v \in \cV$ is a unit vector we easily obtain an orthonormal system  of $L^2_{\varrho}(\cU;\cV) $ of the form $\{ v \psi_i \}_{i \in \bbN}$.
\begin{lemma}
[Holomorphy of order one polynomials]
\label{lm:Holomorphyorderone}
Let $\varrho$ be a tensor-product probability measure as in \eqref{tensor-product-measure}, $p \in (0,1]$, $\bm{b} \in [0,\infty)^{\bbN}$ with $\bm{b} \in \ell^p(\bbN)$, $v \in \cV\setminus \{ 0\}$ and consider a sequence $\bm{c} = (c_i)_{i \in \bbN} \subset \bbR^{\bbN}$ with  $|c_i| \leq b_i$ for all $i \in \bbN$. Define the function
\begin{equation}\label{eq:def_f}
f  = \sum^{\infty}_{i=1} c_i v \psi_i.
\end{equation}
Then $f$ is $(\bm{b},1)$-holomorphic with
\begin{equation*}
\nmu{f}_{L^{\infty}(\cR({\bm{b}});\cV)} \leq \frac{\|v\|_{\cV}}{\sigma} \left ( 1 + (|\tau|+1) \nm{\bm{c}}_{p} \right ).
\end{equation*}
\end{lemma}

\prf{
Notice that $(|{c}_i|)_{i \in \bbN} \in \ell^1(\bbN)$ and that $f$ is holomorphic at any $\bm{y} \in \cU$ for which the series $\sum^{\infty}_{i=1} c_i v \psi_i (\bm{y})$  converges absolutely. Now suppose that $\bm{y} \in \cE_{\bm{\rho}}$, where $\bm{\rho}$ satisfies the condition in  \eqref{def:b-eps-holo} with $\varepsilon = 1$. Then $|y_i| \leq (\rho_i + \rho^{-1}_i)/2$ and therefore
\eas{
\left \| \sum^{\infty}_{i=1} c_i v \psi_i(\bm{y}) \right \|_{\cV} & \leq \sum^{\infty}_{i=1} \|c_i v \|_{\cV} \left ( \frac{|y_i| + |\tau|}{\sigma} \right )
\\
& \leq \frac{\|v\|_{\cV}}{\sigma} \left ( \sum^{\infty}_{i=1} |c_i| \left ( (\rho_i + \rho^{-1}_i)/2 - 1 \right )  + \nm{\bm{c}}_{1}+ |\tau| \nm{\bm{c}}_{1} \right )
\\
& \leq \frac{\|v\|_{\cV}}{\sigma} \left ( 1 + (|\tau|+1) \nm{\bm{c}}_{p}  \right ),
}
as required.
}
 
In the next result, we relate the $m$-width $\theta_m(\bm{b})$  in \eqref{theta-upsilon-known-aniso} to the Gelfand $m$-width of a certain finite-dimensional unit ball.  

\lem{
[Reduction to a discrete problem; known anisotropy case]
\label{l:discrete-reduction-known}
Let $\varrho$ be a tensor-product probability measure as in \eqref{tensor-product-measure}, $\bm{b} \in [0,\infty)^{\bbN}$ with $\bm{b} \in \ell^1(\bbN)$. Let $N \in \bbN$ and $I \subset \bbN$ be an index set with $|I| = N$. Then the $m$-width \eqref{theta-upsilon-known-aniso} satisfies 
\begin{equation}
\label{eq:bound_theta_finite}
\theta_m(\bm{b}) \geq C(\b,\tau,\sigma) \cdot d^m(B^{\infty}_N(\bm{b}_I),\ell^2_N),
\end{equation}
where   $B^{\infty}_N(\bm{b}_I)$ is as in \eqref{eq:defB_weighted} and 
\begin{equation}\label{eq:def_C}
C(\b,\tau,\sigma)=   \frac{{\sigma  }   }{  1+(1+|\tau|)\|\bm{b}\|_1   }.
\end{equation}
}

 {Observe that the bound \eqref{eq:bound_theta_finite} holds trivially when $N \leq m$, since $d^m(B^{\infty}_N(\bm{b}_I),\ell^2_N)=0$ in this case (see \cite[Sec.~10.1]{foucart2013mathematical}). }

\begin{proof}
Let $\cL$ be as in Def.\ \ref{def:L-B}. Then there are $v \in \cV \backslash \{0\}$, $w \in \cV \backslash \{0\}$ and a normed vector space $ \widetilde{\cY} \subseteq L^2_\varrho(\cU)$   such that $\cL(v g) = w \widetilde{\cL}(g)$ whenever $v g \in \cY$ and $g \in L^2_{\varrho}(\cU)$, where $\widetilde{\cL} : \widetilde{\cY} \rightarrow \bbR^m$  is as in Def.\ \ref{def:adaptive-sampling-scalar}. Now let $\bm{c} = (c_i)_{i \in I} \subset \bbR^{\bbN}$ be any sequence supported in $I$ with  $|c_i| \leq b_i$, $\forall i \in I$, and define $f : \cU \rightarrow \cV$  by
\begin{equation}\label{eq:def_fLm4.2}
f =  c v \sum_{i \in I}c_i  \psi_i ,\qquad c = \frac{\sigma}{\nm{v}_{\cV}( 1 + (|\tau|+1) \nm{\bm{b}}_{1}) }.
\end{equation}
Since $\nm{\bm{c}}_1 \leq \nm{\bm{b}}_1$, Lemma \ref{lm:Holomorphyorderone} implies that $f \in \cH(\bm{b}) \subseteq \cY$. By definition, we have
\bes{
\cL(f) = w \widetilde{\cL}(g),\qquad g = c  \sum_{i \in I} c_i \psi_i \in \widetilde{\cY}.
}
Now let $\Gamma : \bbR^{N} \rightarrow \bbR^m$ be the scalar-valued adaptive sampling operator defined by
\bes{
\Gamma(\bm{d}) = \widetilde{\cL} \left ( c \cdot \sum_{i \in I} d_{i} \psi_i \right ),\quad \bm{d}=(d_i)_{i \in I} \in \bbR^{N}. 
}
Here, for convenience, we index vectors in $\bbR^N$ using the index set $I$. We need to show that this operator is well defined. Since $\widetilde{\cL}$ has domain $\widetilde{\cY}$, this is equivalent to showing that   $\sum_{i \in I} d_i \psi_i \in \widetilde{\cY}$ for all $\bm{d}=(d_i)_{i \in I}  \in \bbR^N$. To see this, recall that $\cH(\bm{b}) \subset \cY$ and therefore
\bes{
\cH_0(\bm{b}) : =\{ f : \cU \rightarrow \cV, \text{$(\bm{b},1)$-holomorphic} \} \subseteq \cY,
}
since $\cY$ is a vector space. Now, for any $\bm{d}=(d_i)_{i \in I}  \in \bbR^N$, the function $\sum_{i \in I} d_i\psi_i$ is entire, therefore
\bes{
v \cdot \sum_{i \in I} d_i \psi_i \in \cH_0(\bm{b}) .
}
Hence,   $\sum_{i \in I} d_i \psi_i  \in \widetilde{\cY}$  due to Def.\ \ref{def:L-B}, and therefore $\Gamma$ is well defined.

With this in hand, recall that $\bm{c} \in \bbR^{\bbN}$ is zero outside of the index set $I$. Hence we may consider it as an element of $\bbR^N$ indexed over $I$. Using this and the definition of $\Gamma$, we have $\cL(f) = w \widetilde{\cL}(g) = w \Gamma(\bm{c}) = (w (\Gamma(\bm{c}))_i)^{m}_{i=1} \in \cV^m$. 

Now let $\cT: \cV^m \rightarrow L^2_{\varrho}(\cU;\cV)$ be an arbitrary  reconstruction map and let $\widetilde{\cT}: \bbR^m  \rightarrow L^2_{\varrho}(\cU;\cV)$ be defined by
\begin{equation*}
\widetilde{\cT} (\bm{z}) = \cT ( w \cdot \bm{z}),\quad \forall \bm{z} \in \bbR^m.
\end{equation*}
  Observe that
\begin{equation}\label{eq:def_Ttilda}
\cT ( \cL (f) ) =  \cT ( w \cdot \Gamma(\bm{c})) =  \widetilde{\cT} ( \Gamma(\bm{c})).
\end{equation}
Let $\cV^*$ be the dual space of $\cV$.  From the Hahn--Banach theorem (see \cite[Thm.~3.3]{rudin1991functional})  there exists a linear bounded functional $\phi_v^* \in  \cV^*$  with unit norm such that $\phi_v^*(v)=\|v\|_{\cV}$. Using this fact and the definition of a norm in Banach spaces, for every $\bm{y} \in \cU$, we get 
\begin{align*}
\| f(\bm{y})- \cT \circ \cL(f)(\bm{y})\|_{\cV} 
&= \sup_{\substack{\phi^* \in \cV^*,  \|\phi^*\|_{\cV^*}=1 }} |\ip{\phi^*}{ f(\bm{y})- \cT \circ \cL(f)(\bm{y})}_{\cV^* \times \cV}| \\
& \geq |\ip{\phi_v^*}{ f(\bm{y})- \cT \circ \cL(f)(\bm{y})}_{\cV^* \times \cV}| \\
&  = \left | c' \sum_{i \in I} c_i \psi_i(\bm{y})  - \ip{\phi_v^*}{\cT \circ \cL(f)(\bm{y})}_{\cV^* \times \cV} \right | ,
\end{align*}
where $c' = c \nm{v}_{\cV} = \sigma / (1+(|\tau|+1) \nm{\bm{b}}_1)$.
Then, squaring and integrating  over  $\cU$, we can use Bessel's inequality on the rightmost term to obtain
\eas{
\nmu{f - \cT \circ \cL(f)}^2_{L^2_{\varrho}(\cU;\cV)} 
& \geq  \| c' \sum_{i \in I} c_i \psi_i  - \phi_v^* \circ \cT \circ \cL(f) \|^2_{L^2_{\varrho}(\cU)} 
\\
&  \geq \sum_{j \in I} |  \ip{\psi_j}{c' \sum_{i \in I} c_i \psi_i  - \phi_v^* \circ \cT \circ \cL(f)}_{L^2_{\varrho}(\cU)} |^2.
}
Combining this with \eqref{eq:def_Ttilda} we  obtain
\eas{
\nmu{f - \cT \circ \cL(f)}^2_{L^2_{\varrho}(\cU;\cV)} 
&\geq (c')^2 \sum_{j \in I} | c_j  -  \ip{\psi_j}{  \phi_v^* \circ \cT \circ \cL(f)}_{L^2_{\varrho}(\cU)} / c' |^2
\\
& = (c')^2 \sum_{j \in I} | c_j  - \ip{\psi_j}{  \phi_v^* \circ \widetilde{\cT}(\Gamma(\bm{c})) }_{L^2_{\varrho}(\cU)} /c'|^2.
}
Notice from the last term that every map $\widetilde{\cT} : \bbR^m \rightarrow L^2_{\varrho}(\cU;\cV)$ gives rise to  a mapping  $\bm{z} \mapsto \Delta(\bm{z})$ with $\Delta : \bbR^m \rightarrow \bbR^{N}$  via
\begin{equation}\label{eq:Delta_z}
\Delta(\bm{z}) = \left (   \ip{\psi_i}{\phi_v^* \circ\widetilde{\cT}(\bm{z}) }_{L^2_{\varrho}(\cU)} /c' \right )_{i \in I}= \left ( \int_{\cU} {\psi_i(\bm{y})} \cdot {\phi_v^* \left(\widetilde{\cT}(\bm{z})(\bm{y})\right)}\D \varrho (\bm{\bm{y}}) /c'\right )_{i \in I}.
\end{equation}
Hence we obtain 
\begin{equation*}
\nmu{f - \cT \circ \cL(f)}^2_{{L^2_{\varrho}(\cU;\cV)}} 
\geq (c')^2 \sum_{i \in I} |  c_i-(\Delta(\Gamma(\bm{c})))_{i} |^2 = (c')^2 \nmu{\bm{c}- \Delta({\Gamma(\bm{c})})}^2_2.
\end{equation*}
Thus, we have shown that for any pair $(\cL,\cT)$ and any $f$ of the form \R{eq:def_fLm4.2}, the error $\nmu{f - \cT \circ \cL(f)}_{{L^2_{\varrho}(\cU;\cV)}}$ can be bounded below by a constant times the error $\nmu{\bm{c}- \Delta({\Gamma(\bm{c})})}_2$ for some pair $(\Gamma,\Delta)$. 
Using this, we deduce that
\eas{
\theta_m(\bm{b}) & = \inf \left \{ \sup_{f \in \cH(\bm{b})} \nm{f - \cT(\cL(f))}_{{L^2_{\varrho}(\cU;\cV)}} : \cL : \cY \rightarrow \cV^m\ \text{adaptive},\ \cT : \cV^m \rightarrow {L^2_{\varrho}(\cU;\cV)} \right \}
\\
& \geq c' \inf \left \{ \sup_{\substack{\bm{c} \in \bbR^{N}, \bm{c} \neq \bm{0} \\ |c_i| \leq b_i, \forall i \in I }} \nmu{\bm{c} - \Delta(\Gamma(\bm{c})) }_2 : \Gamma : \bbR^{N} \rightarrow \bbR^m\text{ adaptive},\ \Delta : \bbR^m \rightarrow \bbR^{N} \right \}
\\
& \geq c'   E_{\mathsf{ada}}^{m}(B^{\infty}_N(\bm{b}_I), {\ell^2_N}),
}
where in the final inequality we recall \eqref{eq:def-E_ada}.
The result now follows from Theorem \ref{thm:widths_prop}. 
\end{proof}

  Next, we give a lower bound for the right-hand side of \eqref{eq:bound_theta_finite}.  For this, we make use of Theorem \ref{thm:stesin}, which is due to Stesin \cite{stesin1975aleksandrov}.  

\begin{proof}[Proof of Theorem \ref{thm:known-lower}] We first prove part (a). Let $N \in \bbN$ {with $N>m$} and $I \subset  \bbN$, $|I|=N$ be the index set corresponding to the largest $N$ entries of $\bm{b}$. First, using the duality result Theorem \ref{thm:equal_dm} and then Lemma \ref{lem:widths_2} we obtain
\be{
\label{d-sup-sub-m}
d^m(B^{\infty}_N(\bm{b}_I),\ell^2_N) = d_m(B^2_N,\ell^1_N(1/\bm{b}_I)) =  d_m(B^2_N(\bm{b}_I),\ell^1_N).
}
Here  $1/\bm{b}_I$ is the vector with entries $1/\bm{b}_I=(1/b_i)_{i \in I}$.
Next,  applying  Theorem \ref{thm:stesin} with $p=2$ and $q=1$ we see that the right-hand side of the previous equation satisfies
\begin{equation}\label{eq:max_sum}
d_m(B^2_N(\bm{b}_I),\ell_N^1) = \left ( \max_{\substack{i_1,\ldots,i_{N-m}\in I \\ i_k \neq i_j}} \left ( \sum^{N-m}_{j=1} (b_{i_j})^2 \right )^{-1/2} \right )^{-1}.
\end{equation}
Now, let $\pi: {\bbN} \mapsto \bbN$ be a bijection whose entries are a nonincreasing rearrangement of the vector  $\b$ in nonincreasing order. That is $\b_{\pi}= (b_{\pi(j)})_{j \in \bbN}$ is such that
\begin{equation*}
b_{\pi(1)} \geq b_{\pi(2)} \geq  \cdots b_{\pi(N)}  \geq \ldots \geq 0.
\end{equation*}
Observe that $\b_{I}$ has a one-to-one relation with the first $N$ terms of $\b_{\pi}$. Hence, the maximum in \eqref{eq:max_sum} is achieved  by  the last $N-m$ terms of  {$(\bm{b}_{\pi(i)})_{i=1}^{N}$}. Thus,  
\begin{equation*}
d_m(B^2_N(\bm{b}_I),\ell_N^1)  =  \min_{\substack{i_1,\ldots,i_{N-m} \in I \\ i_k \neq i_j}  } \left ( \sum^{N-m}_{j=1}  (b_{i_j})^2 \right)^{1/2} 
=   \left ( \sum^{N}_{j=m+1} b^{2}_{\pi(j)} \right)^{1/2}.
\end{equation*}
Moreover,  we know that  $\sigma_m(\b)_2^2 = \sum_{j=m+1}^{\infty} b_{\pi(j)}^2  $, which implies that
\bes{
\sigma_m(\b)_2^2  = \sum_{j=m+1}^N b_{\pi(j)}^2  +\sum_{j=N+1}^{\infty} b_{\pi(j)}^2 = (d_m(B^2_N(\bm{b}_I),\ell_N^1))^2  +\sum_{j=N+1}^{\infty} b_{\pi(j)}^2  .
}
Then, from  \R{d-sup-sub-m} and Lemma \ref{l:discrete-reduction-known} we obtain
\begin{align*}
\sigma_m(\b)_2^2  & \leq  C(\b,\tau,\sigma)^{-2} \cdot \theta_m^2(\b) +\sum_{j=N+1}^{\infty} b_{\pi(j)}^2, 
\end{align*}
where   $C(\b,\tau,\sigma)$ as in \eqref{eq:def_C}. Taking limit when $N \rightarrow \infty$ we obtain the result.

 Next we prove part (b). Consider the sequence $\bm{{b}} = ({b}_i)^{\infty}_{i=1}$ with
\begin{equation}\label{eq:def_b_2m}
{b}_i = (2m)^{-1/p}, \quad i=1, \ldots, 2m, \quad {b}_i=0, \, i>2m.
\end{equation}
Observe that $\bm{{b}} \in \ell^p_{\mathsf{M}}(\bbN)$ and  that $\|{\bm{{b}}}\|_{p,\mathsf{M}} = \|{\bm{{b}}}\|_p = 1$  by construction. Also,  note that 
\begin{equation*}
\sigma_m(\bm{{b}})_2 = \sqrt{\sum^{2m}_{i=m+1} (2m)^{-2/p} } = 2^{-1/p} m^{1/2-1/p}.
\end{equation*}
Thus using this and part (a) we get ${\overline{\theta_m}(p,\mathsf{M}) }\geq c \cdot \sigma_{m}(\bm{b})_2  \geq c \cdot 2^{-1/p} m^{1/2-1/p}$, where 
\begin{equation}\label{eq:other_c}
c=\frac{{\sigma  }   }{  1+(1+|\tau|)\|\bm{b}\|_p   } \leq  C(\b,\tau,\sigma).
\end{equation}

Finally, we prove part (c).  {Let $c_{p,g} = \left(\sum_{n \in \bbN}(ng(n))^{-1} \right)^{-1/p}$
and consider the sequence $\bm{b} = (b_i)^{\infty}_{i=1}$ defined by $b_i = c_{p,g} (i g(i))^{-1/p}$. Observe that $\nm{\bm{b}}_p=1$ by construction. Recall that $\ell^p_{\mathsf{M}}(\bbN)$ is the space of sequences whose minimal monotone majorant is in $\ell^p(\bbN)$, with norm defined as the $\ell^p$-norm of the majorant. Since the constructed $\bm{b}$ is mononotonically nonincreasing, it is equal to its minimal monotone majorant. Therefore, $\bm{b} \in \ell^p_{\mathsf{M}}(\bbN)$ and $\|{{\bm{b}}}\|_{p,\mathsf{M}} = \|{\bm{b}}\|_{p} = 1$.}
Then, using {monotonicity once more, we get that}
\begin{equation*}
\begin{split}
(\sigma_m({\bm{b}})_2)^2 = c_{p,g}^2\sum_{i=m+1}^\infty (i g(i))^{-2/p} \geq  c_{p,g}^2\sum_{i=m+1}^{2m} (i g(i))^{-2/p}
 \geq  {c_{p,g}^2 2^{-2/p}}
  \cdot  (g(2m))^{-2/p} m^{1-2/p}.
\end{split}
\end{equation*}
Hence, using  part (a) we get
\begin{equation*}
 \overline{\theta_m}(p,{\mathsf{M}})   
 \geq  c \cdot c_{p,g} 2^{-1/p}
  \cdot  (g(2m))^{-1/p}\cdot  m^{1/2-1/p},
\end{equation*}
where $c>0$ is the constant depending  on $\varrho$ in \eqref{eq:other_c}, as required.
\end{proof}

\subsection{Proof of Theorem \ref{thm:unknown-lower-1}}

We {first} proceed as in the proof of Theorem \ref{thm:known-lower}.  The following  lemma does for the $m$-width $\theta_m(p)$ what Lemma \ref{l:discrete-reduction-known} did for $\theta_m(\bm{b})$.

\lem{
[Reduction to a discrete problem; unknown anisotropy case]
\label{l:discrete-reduction-unknown}
Let  $p \in (0,1]$, $N \in \bbN$, $\varrho$ be a tensor-product probability measure  as in \eqref{tensor-product-measure} and $\tau$, $\sigma$ be as in \eqref{def:alphabeta}. Then the $m$-width \eqref{theta-upsilon-unknown-aniso} {satisfies}
\begin{equation*}
\theta_m(p) \geq C(\tau,\sigma) \cdot d^m(B^{p}_N,\ell^2_N),
\end{equation*}
where   $B^{p}_N$ is as in \eqref{eq:defB_weighted} with $\bm{w} = \bm{1}$ and 
\begin{equation}\label{eq:def_C_unknown}
C(\tau,\sigma)=      \frac{{\sigma  }   }{    2+|\tau|    } .
\end{equation}
}
\begin{proof}
Recall that $\theta_m(p)$ is defined by
\begin{equation*}
\theta_m(p) = 
 \inf \left \{ \sup_{f \in \cH(p)} \|f - \cT ( \cL(f)) \|_{L^2_\varrho (\cU;\cV)} : \cL :\cY \rightarrow \cV^m\text{ adaptive},\ \cT : \cV^m \rightarrow L^2_\varrho (\cU;\cV)\right \}.
\end{equation*}
Let $\cL : \cY\subset L^2_\varrho (\cU;\cV)  \rightarrow \cV^m$ be {a general adaptive sampling operator} as in Def.\ \ref{def:L-B} and $v$ be the corresponding nonzero element of $\cV $.  Consider $\bm{b} \in \ell^p(\bbN)$ with $\bm{b} \in [0,\infty)^{\bbN}$ and $\|\bm{b}\|_{p} \leq 1$,  
and let $\bm{c} = (c_i)_{i \in \bbN} \in \bbR^{\bbN}$ {be any} sequence supported in $[N]$ with  $|c_i| = b_i$ for  {$ i \in [N]$}. Define  the function 
\begin{equation}\label{eq:def_f_unknown}
f =  \frac{\sigma}{ (2+ |\tau|) \nm{v}_{\cV}} v  \sum^N_{i = 1} c_i  \psi_i.
\end{equation}
Lemma~\ref{lm:Holomorphyorderone} {implies} that $f \in \cH(\bm{b})$. We now use the same arguments as in the proof of Lemma \ref{l:discrete-reduction-known}  to obtain
\begin{equation*}
  \nmu{f - \cT \circ \cL(f)}_{L^2_{\varrho}(\cU;\cV)} 
 \geq c'   \nmu{\bm{c}- \Delta(\Gamma(\bm{c}))}_2,
\end{equation*}
where  {$c' = \sigma / (2+|\tau|)$} and {$\Gamma$, $\Delta$ are as before}. We next take the supremum over $\bm{b} \geq \bm{0}$ with $\bm{b} \in \ell^p ({\bbN})$ and $\|\bm{b}\|_{p}  \leq 1$ and all sequences $\bm{c}$ of the above form. Then we get
\bes{
\sup_{f \in \cH(p)}\nmu{f - \cT \circ \cL(f)}_{L^2_{\varrho}(\cU;\cV)} \geq c' \sup_{\substack{\bm{c} \in \ell^p(\bbN),\ \nm{\bm{c}}_p \leq 1 \\ \supp(\bm{c}) \subseteq [N]}} \nm{\bm{c} - \Delta(\Gamma(\bm{c}))}_2.
}
 Hence
\bes{
\theta_m(p) \geq c' \inf \left \{ \sup_{\substack{\bm{c} \in \ell^p(\bbN),\ \nm{\bm{c}}_p \leq 1 \\ \supp(\bm{c}) \subseteq [N]}}\nm{\bm{c} - \Delta(\Gamma(\bm{c}))}_2 : \Gamma : \bbR^{N} \rightarrow \bbR^m\text{ adaptive},\ \Delta : \bbR^m \rightarrow \bbR^{N}  \right \}.
}
Using this and \eqref{eq:def-E_ada} we see that
\bes{
\theta_m(p) \geq c' E^{m}_{\mathsf{ada}}(B^p_N,\ell^2_N).
}
We now apply Theorem \ref{thm:widths_prop}    with $\cK= B_N^p$ and $\cX=\ell^2_N$ to get the result.
\end{proof}

\begin{proof}
[Proof of Theorem \ref{thm:unknown-lower-1}] We first prove part (a). To do so, we use the bound obtained in Lemma \ref{l:discrete-reduction-unknown}.  Let $N \in \bbN$ be such that
\begin{equation}
N \geq m \E^{\frac{\log(3^8\E)}{2p}m-1}   \Leftrightarrow  \frac{ \frac{2p}{\log(3^8\E)}\log(\E N / m)}{m} \geq 1.
\end{equation}
Then, from  Proposition \ref{prop:lowerbound} with $q=2$ we obtain
\begin{equation*}
 d^m(B^{p}_N, \ell^2_N) \geq  \left( \frac{1}{2}\right)^{2/p-1/2}.
 \end{equation*} 
 Thus,
\begin{equation*}
\theta_m(p) \geq \frac{\sigma}{ 2 +|\tau|  }  
\left( \frac{1}{2}\right)^{2/p-1/2},
\end{equation*} 
as required.
Part (b)  follows immediately from part (b) of Theorem \ref{thm:known-lower} and the inequality $\theta_m({p,\mathsf{M}}) \geq \overline{\theta_m}(p,{\mathsf{M}})$. 
\end{proof}

\subsection{Proof of Theorem \ref{thm:unknown-upper-1}}

 As discussed in \S\ref{S:class_f}, polynomial approximation theory for the class  $\cH(\bm{b})$ has been extensively studied. In light of this, we will employ polynomial techniques to establish the two upper bounds presented in Theorems \ref{thm:unknown-upper-1} and \ref{thm:unknown-upper-2}. See Appendix \ref{Ap:poly} for further details on the relevant polynomial approximation theory. 
 We commence in this subsection with the proof of Theorem \ref{thm:unknown-upper-1}.

\begin{proof}[{Proof of Theorem \ref{thm:unknown-upper-1}}]
{The proof is divided into three parts. We first {construct a} sampling operator $\cL$ {and show that it is a well-defined sampling operator in the sense of Def.\ \ref{def:L-B}. Then we prove parts (a) and (b), respectively.}}

Consider  the setup described in \S\ref{ss:leg-poly-exp}. Specifically, let $\cF$ be the set of multi-indices with at most finitely-many zero entries and $\{ \Psi_{\bnu}\}_{\bnu \in \cF}$ be the orthonormal Legendre basis of $L^2_\varrho(\cU)$. Let $S \subset \cF$ be a finite index  of size $|S|=m$ that will be chosen later in the proof and $\cY \supseteq \cH(\bm{b})$ with $\cY \hookrightarrow L^2_{\varrho}(\cU ; \cV)$. We now define   $\cL:  \cY   \rightarrow \cV^m$  and  $\cT:  \cV^m \rightarrow L^2_{\varrho}(\cU;\cV)$ by
\begin{equation}\label{eq:Lf_Tv}
\cL (f) = \left( \ip{f}{\Psi_{\bnu}}_{L^2_{\varrho}(\cU)}\right)_{\bnu \in S} \text{ and }
\cT (\bm{v}) = \sum_{\bnu \in S} v_{\bnu}  \Psi_{\bnu},
\end{equation}
{for any $f \in \cY$ and $\bm{v} { = (v_{\bm{\nu}})_{\bm{\nu} \in S} } \in \cV^m$, respectively.}
{Observe that $\ip{f}{\Psi_{\bnu}}_{L^2_{\varrho}(\cU)} \in \cV$ are precisely the coefficients of the expansion of $f$ in \eqref{f-exp}. However, we keep this notation to emphasize that $\cL$ is a linear operator.}  

We first prove that $\cL$ {is well defined and} {that it satisfies the conditions of} Def.\ \ref{def:L-B}. By construction and the fact that $\cY \hookrightarrow L^2_{\varrho}(\cU ; \cV)$ we readily see that $\cL$ is a bounded linear operator. Therefore, it suffices to show there exists a normed vector space $\widetilde{\cY} \hookrightarrow L^2_\varrho(\cU)$, a nonzero $v \in \cV$, and a bounded, linear scalar-valued sampling operator $\widetilde{\cL}: \widetilde{\cY} \rightarrow \bbR^m$ such that, if $v g \in \cY$ {for $g \in  L^2_{\varrho}(\cU)$}, then $g \in \widetilde{\cY}$ and $\cL(vg)=v\widetilde{\cL}(g)$.
To this end, let $v \in \cV$, $\nm{v}_{\cV} = 1$, be arbitrary and define the space
 \bes{
 \widetilde{\cY} = \{ g \in L^2_{\varrho}(\cU) : v g \in \cY \}.
 }
 It is easily seen that this is a vector space and that the quantity $ \nm{g}_{\widetilde{\cY}} = \nm{v g}_{\cY}$, $\forall g \in \widetilde{\cY}$, defines a norm on $\widetilde{\cY}$.
 Moreover, using  the fact that $\nm{v}_{\cV} = 1$ and $\cY \hookrightarrow L^2_{\varrho}(\cU ; \cV)$, there exists a constant $C>0$ such that
 \bes{
  \nm{g}_{L^2_{\varrho}(\cU)} = \nm{ v g}_{L^2_{\varrho}(\cU ; \cV)} \leq C \nm{v g}_{\cY} = C \nm{g}_{\widetilde{\cY}}, \quad \forall g \in \widetilde{\cY}.
 }
 Hence $\widetilde{\cY} \hookrightarrow L^2_{\varrho}(\cU)$. We now define the scalar-valued sampling operator
 \bes{
 \widetilde{\cL}: \widetilde{\cY} \rightarrow \bbR^m,\quad \widetilde{\cL}(g) = \left( \ip{g}{\Psi_{\bnu}}_{L^2_{\varrho}(\cU)}\right)_{\bnu \in S}, \quad {\forall g \in \widetilde{\cY}}.
}
{Note that $\widetilde{\cL}$ is closely related to the operator $\cL$, except that it is defined over a space $\widetilde{\cY}$ consisting of scalar-valued functions.
Observe that $\widetilde{\cL}$ is} linear and bounded, with the latter property due to the fact that $\widetilde{\cY} \hookrightarrow L^2_{\varrho}(\cU)$. Moreover, if $v g \in \cY$ then $g \in \widetilde{\cY}$ and ${\cL}(vg)=v \widetilde{\cL}(g)$ by construction. Therefore $\widetilde{\cL}$ satisfies the desired properties. We deduce that $\cL$ {is a well defined operator that satisfies the conditions of} Def.\ \ref{def:L-B}, as required.

We now prove part (a) of Theorem \ref{thm:unknown-upper-1}. Let $f \in \cH(\bm{b})$. By \eqref{eq:Lf_Tv} and Parseval's identity we get
\begin{equation*}
 \|f - \cT \circ \cL (f) \|_{L^2_{\varrho}(\cU;\cV)}^2  = \sum_{\bnu \not\in S} {\|c_{\bnu}\|_{\cV}^2},
\end{equation*}
where $\bm{c}=(c_{\bnu})_{\bnu \in \cF}$ is as in \R{f-exp} (see \S\ref{ss:leg-poly-exp} for further information). Now we choose the specific set $S$.  Let $\bm{b} \in \ell^{p}(\bbN), \bm{b} \in [0,\infty)^{\bbN}$. From Corollary \ref{cor:known}, there exists a set $S \subset \cF$ of size $|S| = m$ depending on $\bm{b}$ and $p$ only such that 
\begin{equation*}
\left(\sum_{\bnu \not\in S} {\|c_{\bnu}\|_{\cV}^2} \right)^{1/2} \leq   C(\bm{b},p) \cdot m^{1/2-1/p},
\end{equation*}
where  $C(\bm{b},p)>0$ is  the  constant in Lemma \ref{lem:coeff_known}. Now, taking supremum over $f \in \cH(\bm{b})$ we get 
\begin{equation*}
\sup_{f \in \cH(\bm{b})}\|f - \cT \circ \cL (f) \|_{L^2_{\varrho}(\cU;\cV)} \leq C(\bm{b},p) \cdot m^{1/2-1/p}.
\end{equation*}
Since $\bm{b}$ was arbitrary, this completes the proof of part (a). 

Next we prove part (b). Let $q \in (p,1)$. Then any $\bm{b} \in \ell^p_{\mathsf{M}}(\bbN)$ satisfies $\bm{b} \in \ell^q(\bbN)$ and therefore
\bes{
\theta_m(\bm{b}) \leq C(\bm{b},q) \cdot m^{1/2-1/q}
}
by part (a).
Using this we get
\bes{
\overline{\theta_m}(p,\mathsf{M}) = \sup_{\nm{\bm{b}}_{p,\mathsf{M}} \leq 1} \theta_m(\bm{b}) \leq \sup_{\nm{\bm{b}}_{p,\mathsf{M}} \leq 1} C(\bm{b},q)  \cdot m^{1/2-1/q}.
}
The result now follows from Lemma \ref{cor:supr}.
\end{proof}

{
\begin{remark}
\label{rem:constants-blowup}
As shown in this proof, the constant $c = c_{p,q}$ in part (b) of Theorem \ref{thm:unknown-upper-1} comes from Lemma \ref{cor:supr}. Unfortunately, the dependence of this constant on $p$ and $q$ is unknown, since it depends on an abstract summability criterion (see, e.g., \cite[Lem.~3.29]{adcock2022sparse}) that does not give explicit upper bounds on the norm of the relevant sequence (in this case, the term $\nm{\tilde{\bm{g}}(p)}_{q}$ in \R{eq:Constant_C2}). A first step towards understanding the constant $c_{p,q}$ would involve modifying the proof of this result to provide explicit bounds. However, one can already deduce from the proof of Lemma \ref{cor:supr} that $c_{p,q}$ must blow up as $q \rightarrow p^{+}$, since the sequence $\tilde{\bm{h}}(p) \notin \ell^p(\bbN)$ (see \R{eq:step1}). Therefore, if part (b) of Theorem \ref{thm:unknown-upper-1} were to hold with $q = p$, its proof would require a different technique.
\end{remark}
}

\subsection{Proof of Theorem \ref{thm:unknown-upper-2}}\label{S:proof2.4}

As in the previous proof, we rely on Legendre polynomial expansions. Consider the setup of \S \ref{ss:leg-poly-exp} once more. In the previous proof, we made a judicious choice of index set $S \subset \cF$ depending on the term $\bm{b}$ and used it to define the sampling and reconstruction map that gave the desired bound.
Recall that $\bm{b}$ controls the nature of the anisotropy of the functions in $\cH(\bm{b})$ (see the discussion in \S\ref{S:known-unk}). Therefore, it is expected that a suitable choice of $S$ that obtains the desired error bound should depend on $\bm{b}$. However, in Theorem \ref{thm:unknown-upper-2} we consider the case of unknown $\bm{b}$. Therefore, we must proceed in a different way, in which the sampling operator $\cL$ and the reconstruction map $\cT$ do not rely on a specific choice of $S$. Fortunately, we can circumvent this issue by formulating a compressed sensing problem. Here, by using a suitable set of (random) measurements and a recovery procedure, we can construct an approximation that gives us the desired error bounds. 

In order to formulate a recovery procedure based on compressed sensing, we need to restrict our attention to a finite `search space' of multi-indices that contain all possible sets that yield the desired approximation rate. Unfortunately, the sets $S$ used in the previous proof (which are based on Corollary \ref{cor:known}) can potentially lie anywhere within the infinite cardinality set $\cF$ when one considers arbitrary $\bm{b} \in \ell^p(\bbN)$. Thus, their union is not a finite set. Luckily, this obstacle can be overcome by restricting to sequences in the monotone $\ell^p$ space, i.e., $\bm{b} \in \ell^p_{\mathsf{M}}(\bbN)$, and by utilizing the concept of anchored sets (see, e.g.,  \cite[Sec.~3.9]{adcock2022sparse}). In this case, it is known that when $\bm{b} \in \ell^p_{\mathsf{M}}(\bbN)$ one can obtain the same rate as in Corollary \ref{cor:known} with a set $S$ that is also anchored (see Corollary \ref{cor:anchored_sigma}).
 
This means that we can restrict the search space to one that contains the union of all anchored sets, which has the desirable property that it is itself a finite index set.
Specifically, given $n \in \bbN$, we now define the search space $\Lambda \subset \cF$ as the index set
\be{
\label{HC_index_set_inf}
\Lambda = \Lambda^{\mathsf{HCI}}_{n} = \left \{ \bm{\nu} = (\nu_k)^{\infty}_{k=1} \in \cF : \prod_{k:\nu_{k}\neq 0} (\nu_k + 1) \leq n,\ \nu_k = 0,\ k {\geq} n \right \} \subset \cF.
}
It is known that $\Lambda$ contains all anchored sets of size at most $n$ \cite[Prop.~2.18]{adcock2022sparse}. Note that the size of this index set is bounded by
\be{
\label{N_bound}
N := | \Lambda^{\mathsf{HCI}}_{n} | \leq  \E n^{2 + \log(n{-1})/\log(2)},\quad \forall n \in \bbN.
}
See {\cite[Lem.\ B.5]{adcock2022sparse}} (this result is based on \cite[Thm.\ 4.9]{kuhn2015approximation}). 
Given $\Lambda$ and $f$ with expansion \R{f-exp}, we now define the truncated expansion of $f$ based on the index set $\Lambda$ and its corresponding vector coefficients in $\cV^N$ as
\begin{equation}\label{eq:expan_lambda}
f_{\Lambda} = \sum_{\bnu \in \Lambda} c_{\bnu} \Psi_{\bnu}, \quad \bm{c}_{\Lambda} = (c_{\bnu_j})_{j=1}^N \in \cV^N.
\end{equation}
Here, $\bnu_{1}, \ldots, \bnu_{N}$ is some ordering of the multi-indices in $\Lambda$. 

In order to prove Theorem \ref{thm:unknown-upper-2}, we need some additional setup. Let $\cB(\cV^N, \cV^m)$ be the space of bounded linear operators from $\cV^N$ to $\cV^m$. In the following we consider operators $\bm{A} \in \cB(\cV^N, \cV^m)$  that arise from matrices $\bm{A}=(a_{i,j})_{i,j=1}^{m,N} \in \bbR^{m \times N}$ as follows:
\begin{equation}
\label{def-inducedBLO}
\bm{x} = (x_i)^{N}_{i=1} \in \cV^N \mapsto \bm{A} \bm{x} = \left ( \sum^{N}_{j=1} a_{ij} x_j \right )^{m}_{i=1} \in \cV^m.
\end{equation}
In this proof we also use the  $\ell^p(\cF;\cV)$- and $\ell^p([N];\cV)$-norms,   as  defined in \S\ref{ss:leg-poly-exp}.

\subsubsection{Setup and the vector recovery problem}\label{S:setup-unk}

We now describe the sampling operator and reconstruction map that are used to establish Theorem \ref{thm:unknown-upper-2}. 
Let $\bm{A}=1/\sqrt{m} (a_{i,j})_{i,j=1}^{m,N} \in \bbR^{m \times N}$, where the $a_{i,j}$ are independent Gaussian   random variables with zero mean and variance one. Note that the entries of $\bm{A}$ define the random vector $\bm{r}$ {by allocating each entry of $\bm{A}$  to an entry of $\bm{r}$. Observe that this gives
\be{
\label{r-distn-general}
\bm{r} \sim \cN(0,I_{m N}).
}
}
Now let $\Phi_{i,\bm{r}}: \cU \rightarrow \bbR$, $\Phi_{i,\bm{r}} = 1/\sqrt{m}\sum_{j=1}^N a_{i,j} \Psi_{\bnu_j}$ for $i=1,\ldots, m$, where $\lbrace \Psi_{\bnu}\rbrace_{\bnu \in \cF}$ are the Legendre polynomials (see \S\ref{ss:leg-poly-exp}). We next define the sampling operator $\cL_{\bm{r}}: \cY \subset L^2_{\varrho}(\cU;\cV) \rightarrow \cV^m$  by 
\begin{equation}\label{eq:L}
\cL_{\bm{r}} (f) = \left( \ip{f}{\Phi_{i,\bm{r}}}_{L^2_{\varrho}(\cU)} \right)_{i=1}^m= \bm{f}.
\end{equation}
Using the same argument as in the proof of Theorem \ref{thm:unknown-upper-1} with {$\Phi_{i,\bm{r}}$} in place of $\Psi_{\bnu}$, we deduce that $\cL_{\bm{r}}$ is well-defined and satisfies Def.\ \ref{def:L-B}. 
Now, let $\bm{c}_{\Lambda} \in \cV^N$ {be the vector of} coefficients of $f_{\Lambda}$ in \eqref{eq:expan_lambda}. Then
\begin{equation*}
\cL_{\bm{r}} (f)  = \bm{A} \bm{c}_{\Lambda}.
 \end{equation*} 
Next, we define $\cT_{\bm{r}}: \cV^m \rightarrow L^2_{\varrho} (\cU;\cV)$ as the reconstruction  map
\be{\label{eq:T}
 \cT_{\bm{r}} (\bm{v})= \begin{cases}  \sum_{\bm{\nu} \in \Lambda} \hat{{c}}_{\bnu}   \Psi_{\bnu} & \bm{v} \in \mathrm{Ran}(\bm{A}),
 \\
 \bm{0} & \bm{v} \notin \mathrm{Ran}(\bm{A}),
 \end{cases}
 }
where {$\mathrm{Ran}(\bm{A})$ is the range of  $\bm{A}$}, the vector $\hat{\bm{c}}_{\Lambda} = (\hat{c}_{\bm{\nu}})_{\bm{\nu} \in \Lambda}$ is defined by
 \bes{
 \hat{\bm{c}}_{\Lambda} = \argmin{} \left \{ \nm{\bm{z}}_{2;\cV} : \bm{z} \in M(\bm{v},\bm{A})\right \},
 }
and $ M(\bm{v},\bm{A})$ is the set of minimizers
 \bes{
 M(\bm{v},\bm{A}) = \argmin{\bm{z} \in \cV^N} \nm{\bm{z}}_{1;\cV} \text{ subject to $\bm{A} \bm{z} = \bm{v}$}.
 }
Note that we make this slightly awkward-looking definition for $\cT_{\bm{r}} $ to ensure that it is a well-defined (i.e., single-valued) map. In general, the minimization problem defined above does not have a unique solution. Thus, we choose the solution with the minimal $\ell^2$-norm to enforce uniqueness. Further, the problem has no solution if $\bm{v} \notin \mathrm{Ran}(\bm{A})$. Hence in this case, we simply set $\cT_{\bm{r}} (\bm{v}) = \bm{0}$. 

Note that the composition of this reconstruction map with the sampling operator $\cL_{\bm{r}} $ gives that
\begin{equation}\label{eq:minimizer}
  \cT_{\bm{r}}  (\cL_{\bm{r}}  (f) ) = \sum_{\bnu \in \Lambda} \hat{{c}}_{{\bm{\nu}}} \Psi_{\bnu},\quad \text{where } \hat{\bm{c}}_{\Lambda} \in \mathrm{argmin} \|\bm{z}\|_{1;\cV} \text{ subject to } \bm{Az}= \bm{f},
    \end{equation} 
and $\bm{f}$ is as in \R{eq:L}. Because of this setup, in order to prove Theorem \ref{thm:unknown-upper-2}, we first need to consider properties of matrices  $\bm{A} \in \bbR^{m \times N}$, and consequently operators in $\cB(\cV^N, \cV^m)$,  to recover (approximately) sparse vectors by solving the (Hilbert-valued) basis pursuit (BP) problem 
\begin{equation}\label{eq:problemQCBP}
\min_{\bm{z} \in \cV^N} \|\bm{z}\|_{1;\cV} \text{ subject to } \bm{Az}= \bm{f},
\end{equation}
with $\bm{f} \in \cV^m$. Specifically, we shall make use of the following property.

\begin{definition}
A matrix $\bm{A} \in \bbR^{m \times N}$ satisfies the \textit{robust Null Space Property (rNSP)} of order $1 \leq s \leq N$ over $\cV^N$ with constants $0 < \rho < 1$ and ${\gamma} > 0$ if
\bes{
\nm{ \x_{S}}_{2; \cV} \leq  \dfrac{\rho \nm{ \x_{S^c}}_{1;\cV} }{\sqrt{s}}+{\gamma} \nm{\bm{A} \x}_{2; \cV} ,\quad \forall \x \in \cV^N,
}
for any $S \subseteq [N]$ with $|S| \leq s$.
\end{definition}
Recall that $\bm{x}_S$ is the vector in $\cV^N$ with $i$th entry equal to $x_i$ if $i \in S$ and zero otherwise.
The following result is a straightforward application of \cite[Prop.\ 4.2]{dexter2019mixed}. It shows that  the rNSP is sufficient to provide an error bound for minimizers of the   BP problem \eqref{eq:problemQCBP}.   
\begin{theorem}\label{Thm:rNSP_recover}
Suppose $\bm{A} \in \bbR^{m \times N}$ has the rNSP over $\cV$ of order $s \in [N]$ with constants $0 < \rho < 1$ and $\gamma >0$. Let $\bm{x} \in \cV^N$, $\bm{f}= \bm{Ax} \in \cV^m$. Then every minimizer $\bm{\hat{x}} \in \cV^N$ of the  BP problem \eqref{eq:problemQCBP} satisfies
\begin{equation}\label{eq:rNSP}
\|\bm{{x}} -\bm{\hat{x}} \|_{1;\cV} \leq C_1 \sigma_{s}(\bm{{x}})_{1;\cV} \text{ and } 
\|\bm{{x}} -\bm{\hat{x}} \|_{2;\cV} \leq C_2 \dfrac{\sigma_{s}(\bm{{x}})_{{1};\cV} }{\sqrt{s}},
\end{equation}
where $C_1=2(1+\rho)/(1-\rho)$, $C_2 = 2(1+\rho)^2/(1-\rho)$.
\end{theorem}
 Here $\sigma_s(\bm{x})_{p;\cV}$ is the $\ell^p$-norm best $s$-term approximation error, as defined in \eqref{def:sigmaV}. Now consider a matrix $\bm{A}=1/\sqrt{m} (a_{i,j})_{i,j=1}^{m,N}$, whose entries $a_{i,j}$ are independent Gaussian random variables with zero mean and variance one. Then, following similar arguments to those in \cite[Ch.~6]{adcock2022sparse} for the scalar case (or \cite{adcock2021deep} for Hilbert-valued case), one can show that if
\begin{equation}\label{eq:measu_poly}
m \geq c \cdot \left( s \cdot \log (2N/s)+ \log(2\epsilon^{-1}) \right),
\end{equation}
for some universal constant $c > 0$, then $\bm{A}$ satisfies rNSP over $\cV$ of order $s$ with constants $\rho \leq 1/2$ and $\gamma \leq 3/2$ with probability at least $1-\epsilon$ (see  \cite[Thm.~6.11--6.12]{adcock2022sparse} for more details).

\begin{remark}
{We note in passing that the universal constant $c$ in \eqref{eq:measu_poly} can be estimated. However, its precise value plays a minor role in what follows. The discussion in \cite[Sec.~6.3.2]{adcock2022sparse} and the estimates in \cite[Rmk. 9.28]{foucart2013mathematical} suggests that in our case this constant can be taken as $c \approx 80.098 \cdot (2 \sqrt{2}+1)^2$.}

\end{remark}

\subsubsection{Error bounds in probability}
To prove Theorem \ref{thm:unknown-upper-2}, we need to show an error bound in expectation. The first step towards doing this is establishing an error bound in probability. This is given by the following theorem.
\begin{theorem}\label{Thm:Prob}
Let $m \geq 3$,  $0<\epsilon<1$,  $\varrho$ be the uniform probability measure on $\cU$, $\cL_{\bm{r}}: \cY \rightarrow \cV^m$ and $\cT_{\bm{r}}: \cV^m \rightarrow L^2_{\varrho}(\cU;\cV)$ be defined as in \eqref{eq:L} and \eqref{eq:T}, respectively,  where $\Lambda = \Lambda^{\mathsf{HCI}}_{n}$ is the index set  in \eqref{HC_index_set_inf} with $n= \lceil m/\log^2(m) \rceil$, and $L=L(m,\epsilon) = \log^2(m) + \log(2\epsilon^{-1})$.   Then the following hold.
\begin{itemize}
\item[a)]  With probability at least $1-\epsilon$,
\begin{equation*}
\sup_{f \in \cH(\bm{b})} \|f- \cT_{\bm{r}} (\cL_{\bm{r}}  (f) ) \|_{L^2_{\varrho}(\cU; \cV)}  \leq \widetilde{C}(\bm{b},p) \cdot \left ( \frac{m}{ L} \right )^{1/2-1/p},
\end{equation*}
 for all $\bm{b} \in \ell^p_{\mathsf{M}}(\bbN)$, $\bm{b} \geq \bm{0}$, and $0 < p < 1$.
\item[b)] With probability one,
\begin{equation*}
\sup_{f \in \cH(\bm{b})}  \|f- \cT_{\bm{r}}  (\cL_{\bm{r}}  (f) ) \|_{L^2_{\varrho}(\cU; \cV)}  \leq  \widetilde{C}(\bm{b},p),
\end{equation*}
for all $\bm{b} \in \ell^p(\bbN)$, $\bm{b} \geq \bm{0}$, and $0 < p < 1$.
\end{itemize}
Here the constant $\widetilde{C}(\bm{b},p)$  depends on $\bm{b}$ and $p$ only.
\end{theorem}

\begin{proof}
We first prove  part (b). Using triangle inequality and Parseval's identity  we obtain
 \begin{equation*}
 \|f- \cT_{\bm{r}} (\cL_{\bm{r}} (f) ) \|_{L^2_{\varrho}(\cU; \cV)} 
 \leq   \|f  \|_{L^2_{\varrho}(\cU; \cV)} + \|  \cT_{\bm{r}} (\cL_{\bm{r}} (f) ) \|_{L^2_{\varrho}(\cU; \cV)} 
 = \|\bm{c} \|_{2;\cV}+\|\hat{\bm{c}}_{\Lambda} \|_{2;\cV}, \quad  \forall f \in \cH(\bm{b}).
 \end{equation*}
Using {the inequality $\|\bm{c} \|_{2;\cV} \leq \|\bm{c} \|_{1;\cV}$ and the fact that $\hat{\bm{c}}_{\Lambda} $} is a minimizer of  \eqref{eq:minimizer}, we get
 \begin{equation*}
\|\hat{\bm{c}}_{\Lambda} \|_{2;\cV} \leq \|\hat{\bm{c}}_{\Lambda} \|_{1;\cV}
\leq \|\bm{c}_{\Lambda} \|_{1;\cV} \leq \|\bm{c} \|_{1;\cV}.
 \end{equation*}
Let $\bm{b} \in \ell^p(\bbN)$, $\bm{b} \geq \bm{0}$, and $f \in \cH(\bm{b})$.  Lemma~\ref{lem:coeff_known} {implies that} the Legendre coefficients satisfy
 \begin{equation*}
 \|\bm{c} \|_{1;\cV} \leq \|\bm{c} \|_{p;\cV}     \leq C(\bm{b},p),
 \end{equation*}
 where   $C(\bm{b},p)$ is the constant {of this lemma}. 
 Therefore, taking the supremum over $f \in  \cH(\bm{b})$ we get 
\begin{equation*}
\sup_{f \in \cH(\bm{b})}  \|f- \cT_{\bm{r}} (\cL_{\bm{r}} (f) ) \|_{L^2_{\varrho}(\cU; \cV)}  \leq {2 {C}(\bm{b},p)},
\end{equation*}
as required.

We now prove  part (a). Let  $\bm{A} \in \bbR^{m \times N}$ be the random matrix used in the construction of $\cL_{\bm{r}}$ and recall that $\bm{A}$ extends to a bounded linear operator  $\bm{A}: \cV^N \rightarrow \cV^m$ given by \eqref{def-inducedBLO}. Let 
\begin{equation*}
s = \left \lfloor  \frac{m}{18c\cdot L(m,\epsilon)}\right\rfloor,
\end{equation*}
where $c>0$ is the constant in \eqref{eq:measu_poly}. {We now assume that $s \geq 1$, which is equivalent to $m/(18c\cdot  L(m,\epsilon)) \geq1$. We show part (a) with the assumption  $s<1$ at the end of the proof.} 
From the estimate for $N$ in  \eqref{N_bound} we get{
\begin{align*}
\log(2N/s) \leq \log(\E N) &\leq \log(\E^2n^{2+\log(n)/\log(2)})  \leq \left(2+\frac{\log(n)}{\log(2)} \right) \log(\E n) \leq 2 (\log(\E n))^2 .
\end{align*}
Hence
\begin{align*}
s \cdot \log(2N/s) + \log (2 \epsilon^{-1}) 
&\leq s \cdot (  2 (\log(\E n))^2+ \log (2 \epsilon^{-1}) ) \\
&\leq 18s \cdot ( \log^2(m) + \log (2 \epsilon^{-1}) ) \\
&\leq 18s \cdot L(m,\epsilon).
\end{align*}
Here we use $n \leq 2m$ in the second inequality   combined with $\log(2\E m) \leq 3 \log(m)$ (since $m \geq 3$) in the second inequality.}
Now, using  \eqref{eq:measu_poly} and the definition of $s$, we deduce that,   with probability $1-\epsilon$   the matrix $\bm{A}$ has the rNSP over $\cV$ of order $s$ with constants $\rho \leq 1/2$ and $\gamma \leq 3/2$. 

Next, we derive a bound for the approximation error.  Using the triangle inequality we  get
 \begin{equation*}
 \|f- \cT_{\bm{r}}  (\cL _{\bm{r}} (f) ) \|_{L^2_{\varrho}(\cU; \cV)} 
 \leq  \|f- f_{\Lambda} \|_{L^2_{\varrho}(\cU; \cV)}  +\|f_{\Lambda}- \cT_{\bm{r}}  (\cL_{\bm{r}}  (f) ) \|_{L^2_{\varrho}(\cU; \cV)}.
 \end{equation*}
 Now, Parseval's identity  gives 
 \begin{equation*}
 \|f_{\Lambda}- \cT_{\bm{r}} (\cL_{\bm{r}} (f) ) \|_{L^2_{\varrho}(\cU; \cV)} = \|\bm{c}_{\Lambda}-\hat{\bm{c}}_{\Lambda} \|_{2;\cV},
 \end{equation*}
 where $\hat{\bm{c}}_{\Lambda}$ is as in \eqref{eq:minimizer}.  Since $\bm{A}$ has the rNSP over $\cV$, we may apply Theorem \ref{Thm:rNSP_recover}   to get
\begin{equation*}\label{eq:sigma_A_1}
\|f_{\Lambda}- \cT_{\bm{r}} (\cL_{\bm{r}} (f) ) \|_{L^2_{\varrho}(\cU; \cV)}  {\leq C_2} \dfrac{\sigma_s(\bm{c}_{\Lambda})_{1;\cV}}{\sqrt{s}}  {\leq 9 \dfrac{\sigma_s(\bm{c}_{\Lambda})_{1;\cV}}{\sqrt{s}}}.
\end{equation*}
This last inequality is due to the fact that $\rho \leq 1/2$ in Theorem \ref{Thm:rNSP_recover}.
Then,   using Corollary \ref{cor:known} with $q=1>p$ we get
\begin{equation*}
\dfrac{\sigma_s(\bm{c}_{\Lambda})_{1;\cV}}{\sqrt{s}} \leq \dfrac{\sigma_s(\bm{c})_{1;\cV}}{\sqrt{s}}  \leq C(\bm{b},p) \cdot s^{1/2-1/p},\quad \forall f \in \cH(\bm{b}),\ \bm{b} \in \ell^p(\bbN),\ \bm{b} \geq \bm{0}.
\end{equation*}
{Now let $\bm{b} \in \ell^p_{\mathsf{M}}(\bbN)$, $\bm{b} \geq \bm{0}$, $f \in \cH(\bm{b})$ and consider the term $\|f- f_{\Lambda} \|_{L^2_{\varrho}(\cU; \cV)} $. By applying Corollary \ref{cor:anchored_sigma} with $q=2>p$, we get that there exists an anchored set $S_{\mathsf{A}} \subset \cF$ with $|S_{\mathsf{A}}| \leq n$ such that 
\bes{
\|\bm{c}-\bm{c}_{S_{\mathsf{A}}}\|_{2;\cV}  \leq      C_{\mathsf{A}}(\bm{b},p) \cdot n^{1/2-1/p},
}
}
where the constant ${C_{\mathsf{A}}}(\bm{b},p) $ is as in \eqref{eq:coeff_known_anch}.
Using the definition of $\Lambda$ given in \eqref{HC_index_set_inf}, we know that it contains the union of all anchored sets of size $n$ {and therefore $S_{\mathsf{A}} \subseteq \Lambda$}. This yields
\begin{equation*}
\|f- f_{\Lambda} \|_{L^2_{\varrho}(\cU; \cV)} = \|\bm{c}-\bm{c}_{\Lambda}\|_{2;\cV}    \leq   \|\bm{c}-\bm{c}_{S_{\mathsf{A}}}\|_{2;\cV}    
\leq      C_{\mathsf{A}}(\bm{b},p) \cdot n^{1/2-1/p},\quad \forall f \in \cH(\bm{b}).
\end{equation*}
Combining these two results, noticing that $s \leq n$ and {using} the definition of {$s$} we get  
 \begin{equation}\label{eq:bound_f_b}
 \|f- \cT_{\bm{r}} (\cL_{\bm{r}} (f) ) \|_{L^2_{\varrho}(\cU; \cV)} 
 \leq \widetilde{C}(\bm{b},p) \cdot \left( \dfrac{m}{L}\right)^{1/2-1/p}, \quad \forall f \in \cH (\bm{b})
 \end{equation}
 and all $\bm{b} \in \ell^p_{\mathsf{M}}(\bbN)$, $\bm{b} \geq \bm{0}$, as required. Here the constant  
\begin{equation}\label{eq:consta_1}
\widetilde{C}(\bm{b},p) = {9 (18c)^{1/p-1/2} C(\bm{b},p) + C_{\mathsf{A}}(\bm{b},p)},
 \end{equation} 
 depends on $\bm{b}$  and $p$ only.

{Finally, we prove the case $s < 1$, i.e., $m < 18c \cdot L(m,\epsilon)$. Since $p<1$, we  have 
\begin{equation*}
1 < \left( \dfrac{m}{18c \cdot L} \right)^{1/2-1/p}.
\end{equation*}
From part (b) we deduce that
\begin{equation*}
\sup_{f \in \cH(\bm{b})}  \|f- \cT_{\bm{r}} (\cL_{\bm{r}} (f) ) \|_{L^2_{\varrho}(\cU; \cV)}  \leq 2 {C}(\bm{b},p)\left( \dfrac{m}{18c \cdot L} \right)^{1/2-1/p}=(18 c)^{1/p-1/2} 2 {C}(\bm{b},p)\left( \dfrac{m}{ L} \right)^{1/2-1/p},
\end{equation*}
where $C$ is the constant from  Lemma~\ref{lem:coeff_known} and $c>0$ the constant in \eqref{eq:measu_poly}.  The result now follows immediately by noticing that $(18 c)^{1/p-1/2} \cdot 2 C \leq \widetilde{C}$, where $\widetilde{C}$ is as in  \eqref{eq:consta_1}.} 
\end{proof}

\begin{proof}
[Proof of Theorem \ref{thm:unknown-upper-2}] 
Let $\cL_{\bm{r}}$, $\cT_{\bm{r}}$ be as in Theorem \ref{Thm:Prob} and set
\begin{equation}\label{eq:def_eps}
\epsilon = \left(\dfrac{m}{\log^2(m)} \right)^{1/2-1/p} <1.
\end{equation}
{Recall that the random vector $\bm{r}$  is defined by \R{r-distn-general}, where $N = |\Lambda|$. Since $\Lambda = \Lambda^{\mathsf{HCI}}_n$ and $n = \lceil m / \log^2(m) \rceil$, we deduce that there exists an $\ell = \ell(m)$ depending on $m$ only such that $\bm{r} \sim \cN(0,I_{\ell})$. In particular, $\ell(m) = m \cdot | \Lambda^{\mathsf{HCI}}_n |$.}
{Next, using the definition of  $\theta_m(p,{\mathsf{M}})$ in \eqref{theta-upsilon-unknown-aniso}, notice that the following holds:
 \begin{equation*}
\theta_m(p,{\mathsf{M}}) \leq \bbE_{\bm{r} \sim \cN(0,I_\ell) }  \sup_{f \in \cH(p,\mathsf{M})} \nmu{f - \cT_{\bm{r}}(\cL_{\bm{r}}(f))}_{L^2_{\varrho}(\cU ; \cV)}.
\end{equation*}
}
We now prove part (a). Let $L(m,\epsilon)$ be as in Theorem~\ref{Thm:Prob}. Since $m \geq 3$, we have $\log(m) > 1$. Therefore
\begin{equation*}
{L(m,\epsilon) \leq \log^2(m) + \log(2) + (1/p-1/2) \log(m) \leq (1/2 + \log(2) + 1/p) \log^2(m) = : c_p \log^2(m)},
\end{equation*}
  where $\epsilon$  is as in \R{eq:def_eps}. 
Now let $E$  be the event 
\begin{equation*}
E = \left \{ \sup_{f \in \cH(\bm{b})} \|f- \cT_{\bm{r}} (\cL_{\bm{r}} (f) ) \|_{L^2_{\varrho}(\cU; \cV)}  \leq \overline{C}(\bm{b},p) \cdot \left ( \frac{m}{\log^2(m)} \right )^{1/2-1/p},\ \forall \bm{b} \in \ell^p_{\mathsf{M}}(\bbN), \bm{b} \geq \bm{0}, 0 < p < 1 \right \},
\end{equation*}
where $\overline{C}(\bm{b},p) = (c_p)^{1/p-1/2} \cdot \widetilde{C}(\bm{b},p)$ and $ \widetilde{C}(\bm{b},p)$ is the constant in Theorem \ref{Thm:Prob}. Then Theorem~\ref{Thm:Prob}   implies that $\bbP(E^c) \leq \epsilon$. Hence
\begin{equation*}
\begin{split}
& \bbE_{\bm{r} \sim \cN(0,I_\ell)}  \sup_{f \in \cH(\bm{b})}\|f- \cT_{\bm{r}} (\cL_{\bm{r}} (f) ) \|_{L^2_{\varrho}(\cU; \cV)}  \\
&=  \bbE \left( \sup_{f \in \cH(\bm{b})} \|f- \cT _{\bm{r}}(\cL _{\bm{r}}(f) ) \|_{L^2_{\varrho}(\cU; \cV)}\Big|E\right) 
+ \bbE \left( \sup_{f \in \cH(\bm{b})} \|f- \cT _{\bm{r}}(\cL _{\bm{r}}(f) ) \|_{L^2_{\varrho}(\cU; \cV)}\Big|E^{c}\right) \\
&\leq  \overline{C}(\bm{b},p) \cdot \left( \dfrac{m}{\log^2(m)}\right)^{1/2-1/p} 
+ {2}\widetilde{C}(\bm{b},p)\cdot \epsilon
\\
& \leq {3}  \overline{C}(\bm{b},p) \cdot \left( \dfrac{m}{\log^2(m)}\right)^{1/2-1/p} .
\end{split}
\end{equation*}
This completes the proof of part (a).

We now prove part (b).
Let   $\cL_{\bm{r}}$, $\cT_{\bm{r}}$ be as in Theorem \ref{Thm:Prob} with $\epsilon$ as in \eqref{eq:def_eps}. Using Theorem~\ref{Thm:Prob} part (a) we get that $\bbP(\widetilde{E}) \geq 1-\epsilon$, where
\bes{
\widetilde{E} = \bigcap_{\substack{\bm{b} \in \ell^q_{\mathsf{M}}(\bbN), \bm{b} \geq \bm{0} \\ 0 < q < 1}} \left \{ \sup_{f \in \cH(\bm{b})} \|f- \cT_{\bm{r}}  (\cL_{\bm{r}} (f) ) \|_{L^2_{\varrho}(\cU; \cV)}  \leq \overline{C}(\bm{b},q) \cdot \left ( \frac{m}{\log^2(m)} \right )^{1/2-1/q}  \right \},
}
and 
\begin{equation*}
\overline{C}(\bm{b},q) =  (c_q)^{1/q-1/2} \cdot \widetilde{C}(\bm{b},q) = 9\left(18 \cdot c \cdot \left(\frac12+\log(2)+\frac1q\right) \right)^{1/q-1/2}\cdot \left ( C(\bm{b},q)  + C_{\mathsf{A}} (\bm{{b}},q) \right ),
\end{equation*}
where   $c$ is the constant in \eqref{eq:measu_poly}.

Note that $C(\bm{b},q)$ is the constant in Lemma \ref{lem:coeff_known} and  $ C_{\mathsf{A}}({\bm{b}},q)$  is the constant in Lemma \ref{lem:coeff_known_anch}, both with $q$ instead of $p$. 
Then, since $q>p$,  applying  Lemma \ref{cor:supr} and Lemma \ref{cor:supr_Anchored} we obtain the following uniform bound for  $\overline{C}(\bm{b},q)$:
\begin{equation*}
 {\sup_{\|\bm{b}\|_{p, \mathsf{M}} \leq 1}  \overline{C}(\bm{b},q) \leq  9\left(18 \cdot c \cdot \left(\frac12+\log(2)+\frac1q\right) \right)^{1/q-1/2}\cdot \left( \sup_{\|\bm{b}\|_{p, \mathsf{M}} \leq 1}C(\bm{b},q)  +  \sup_{\|\bm{b}\|_{p, \mathsf{M}} \leq 1} C_{\mathsf{A}}(\bm{{b}},q)  \right)\leq c_{p,q},}
\end{equation*}
where $c_{p,q}$ is a positive constant depending  on $p$ and $q$ only. Next, recall that
\begin{equation*}
\cH(p,\mathsf{M}) = {\bigcup}    \{ \cH(\bm{b}) : \bm{b} \in \ell^p_{\mathsf{M}}(\bbN),\ \bm{b} \geq \bm{0},\ \nm{\bm{b}}_{p,\mathsf{M}} \leq 1 \}.
\end{equation*}
Since $q > p$, any $\bm{b} \in \ell^p_{\mathsf{M}}(\bbN)$ satisfies $\bm{b} \in \ell^q_{\mathsf{M}}(\bbN)$. Therefore
\eas{
\widetilde{E} & \subseteq \bigcap_{\substack{\bm{b} \in \ell^p_{\mathsf{M}}(\bbN), \bm{b} \geq \bm{0} \\ \nm{\bm{b}}_{p,\mathsf{M}} \leq 1 \\ 0 < p< q < 1}} \left \{ \sup_{f \in \cH(\bm{b})} \|f- \cT_{\bm{r}}  (\cL_{\bm{r}}  (f) ) \|_{L^2_{\varrho}(\cU; \cV)}  \leq \overline{C}(\bm{b},q) \cdot \left ( \frac{m}{\log^2(m)} \right )^{1/2-1/q}  \right \}
\\
& \subseteq \bigcap_{\substack{\bm{b} \in \ell^p_{\mathsf{M}}(\bbN), \bm{b} \geq \bm{0} \\ \nm{\bm{b}}_{p,\mathsf{M}} \leq 1 \\ 0 < p< q < 1}} \left \{ \sup_{f \in \cH(\bm{b})} \|f- \cT_{\bm{r}}  (\cL_{\bm{r}}  (f) ) \|_{L^2_{\varrho}(\cU; \cV)}  \leq c_{p,q} \cdot \left ( \frac{m}{\log^2(m)} \right )^{1/2-1/q}  \right \}
 = : E.
}   
Therefore $\bbP(E) \geq \bbP(\widetilde{E}) \geq 1-\epsilon$ and by using a similar argument to that of part (a) we obtain
\begin{equation*}
\begin{split}
& \bbE_{\bm{r} \sim \cN(0,I_\ell) }  \sup_{f \in \cH(p,{\mathsf{M}})}\|f- \cT_{\bm{r}}  (\cL_{\bm{r}}  (f) ) \|_{L^2_{\varrho}(\cU; \cV)}  \\
&=  \bbE \left(\sup_{f \in \cH(p,{\mathsf{M}})} \|f- \cT_{\bm{r}}  (\cL_{\bm{r}}  (f) ) \|_{L^2_{\varrho}(\cU; \cV)} \Big|E\right) 
+ \bbE \left( \sup_{f \in \cH(p,{\mathsf{M}})} \|f- \cT_{\bm{r}}  (\cL_{\bm{r}}  (f) ) \|_{L^2_{\varrho}(\cU; \cV)} \Big|E^{c}\right) \\
&\leq c_{p,q}   \cdot \left( \dfrac{m}{\log^2(m)}\right)^{1/2-1/q} 
+ {2} c_{p,q}  \cdot \epsilon \quad \leq  \quad  {3} c_{p,q}   \cdot \left( \dfrac{m}{\log^2(m)}\right)^{1/2-1/q} ,
\end{split}
\end{equation*}
where $c_{p,q} $ is a constant depending on $p$ and $q$ only. To complete the proof, we need to remove the $\log^2(m)$ factor from the error bound. However, this follows inmediately (up to a possible change in the constant) since $q\in (p,1)$ was arbitrary.
\end{proof}

\section{Conclusions and open problems}\label{S:conclusions}

This work introduced new theoretical guarantees for the limits of approximating Banach-valued, $(\bm{b},1)$-holomorphic functions in infinite dimensions from finite data, where $\bm{b} \in \ell^p(\bbN)$, $0<p<1$. Specifically, in Theorems \ref{thm:known-lower}--\ref{thm:unknown-upper-2} we established new lower and upper bounds for various (adaptive) $m$-widths in both the known and unknown anisotropy cases. As we showed, when these bounds are decaying, they do so at a rate close to $m^{1/2-1/p}$.
In our main contributions  we demonstrated that optimal recovery is attainable for functions in the monotone case, i.e., $\bm{b} \in \ell^p_{\mathsf{M}}(\bbN)$, while approximation from data is impossible without any prior information on the variables, i.e. in the space $\cH(p)$, in which $\bm{b}$ is unknown and $\bm{b} \in \ell^p(\bbN)$ only.  Note that this work does not consider algorithms achieving these rates, but `reconstruction maps' involving minimizers of certain convex optimization problems. For a more detailed exploration using algorithms to achieve rates of the form $(m/\mathrm{polylog}(m))^{1/2-1/p}$ we refer to \cite{adcock2022efficient}.

This aside, there are several other promising directions for future research.
First,   as mentioned there are a variety of methods that can achieve an approximation error decay rate of $(m/\mathrm{polylog}(m))^{1/2-1/p}$ for $(\bm{b},\epsilon)$-holomorphic functions. For instance, in \cite{adcock2022nearoptimal,adcock2022efficient} the authors achieve an upper bound for the approximation error  in the Hilbert-valued case for each $\bm{b} \in \ell^p_{\mathsf{M}}(\bbN)$ using i.i.d. pointwise samples, i.e., \textit{standard information} as it is commonly termed \cite{novak2008trac,novak2010trac}. This work belongs to the class of nonuniform guarantees, since the decay rates obtained from i.i.d. samples hold for a fixed function $f$. On the other hand, the results shown in this paper are uniform, and therefore stronger, since they hold  for any function belonging to the given class. However, our upper bounds do not consider pointwise samples. It is still an open problem to prove upper bounds using i.i.d. pointwise samples in  the uniform case. 

Second, as shown in part (b) of Theorems \ref{thm:unknown-upper-1}  and \ref{thm:unknown-upper-2},   our upper bounds for $\overline{\theta_m}(p,\mathsf{M})$ and $\theta_m(p,\mathsf{M})$ are non-sharp in comparison to the lower bounds shown in Theorems \ref{thm:known-lower} and \ref{thm:unknown-lower-1} by an algebraic factor that can be made arbitrarily small. We conjecture that this gap can be closed. A possible route towards doing so involves showing that the constant $C(\bm{b},p)$ in Lemma \ref{cor:supr} and $C_{\mathsf{A}}(\bm{b},p)$  in Lemma \ref{cor:supr_Anchored} can be bounded uniformly for $\bm{b} \in \ell^p_{\mathsf{M}}(\bbN)$ with $\|\bm{b}\|_{p,\mathsf{M}}\leq 1$.

Third, in the case of known anisotropy, Theorem \ref{thm:known-lower} part (b) establishes that approximation from finite data is possible, even for the worst case of $\bm{b} \in \ell^p(\bbN)$ with unit norm, i.e., there is a lower bound for $\overline{\theta_m}(p)$ with a rate of $m^{1/2-1/p}$. Nonetheless, Theorem \ref{thm:unknown-upper-1} part (b) does not provide an analogous upper bound for $\overline{\theta_m}(p)$. As a result, it is interesting to bridge this gap by either demonstrating that $\overline{\theta_m}(p)$ decreases at a rate of $m^{1/2-1/p}$ as $m \rightarrow \infty$ or by establishing that it remains nondecreasing, such as $\theta_m(p)$ in the unknown anisotropy case in Theorem \ref{thm:unknown-lower-1} part(a).

Fourth, it is an open problem to extend  Theorem \ref{thm:unknown-upper-2} to the Banach-valued case. While the analysis in \cite{adcock2022nearoptimal} indicates that such an extension is possible, using it would result in a worse exponent $1/2(1/2-1/p)$ (or $1/2(1/2-1/q)$) in place of $1/2-1/p$ (or $1/2-1/q$), which is suboptimal. Proving that the same rates can be achieved in the Banach-valued case would close a key gap observed in \cite{adcock2022nearoptimal} between approximating holomorphic Hilbert- and Banach-valued functions.

Finally, it remains an open problem to derive bounds in the $L^{\infty}(\cU ; \cV)$-norm instead of the $L^2_{\varrho}(\cU ; \cV)$-norm. The approximation theory of parametric PDEs is well studied in the $L^{\infty}(\cU ; \cV)$-norm. Yet our proof strategy for the lower bounds on the $m$-widths relies on using the $L^2_{\varrho}(\cU ; \cV)$-norm.
We note that the best $s$-term polynomial approximation attains rates of the form $s^{1-1/p}$ in this norm. We therefore conjecture that versions of our main theorems hold in this norm, except with an exponent of $1-1/p$ (or $1-1/q$) in place of $1/2-1/p$ (or $1/2-1/q$).

\section*{Acknowledgments}
The authors would like to thank Aaron Berk and Simone Brugiapaglia for helpful discussions. {They would also like to thank the anonymous reviewers for their useful comments and suggestions.}
BA acknowledges the support of the Natural Sciences and Engineering Research Council of Canada of Canada (NSERC) through grant RGPIN-2021-611675.

\appendix

\section{Summability and best $s$-term polynomial approximation rates}\label{Ap:poly}

This appendix presents four lemmas on the summability of the coefficients of the Legendre polynomial expansion for the class of $(\bm{b},1)$-holomorphic functions introduced in  \S\ref{S:class_f}.  These will allow us to obtain upper bounds on the $m$-widths \eqref{theta-upsilon-known-aniso}--\eqref{theta-upsilon-unknown-aniso}.

\subsection{Legendre polynomial expansions}
\label{ss:leg-poly-exp}

First, we require some additional notation. Let $\bbN_0^{\bbN}$ be the set of infinite multi-indices. For a multi-index $\bm{\nu}=(\nu_k)_{k \in \bbN} \in \bbN_0^{\bbN}$ we let
\begin{equation*}
\|\bm{\nu}\|_0 = |\supp (\bm{\nu}) |,
\end{equation*}
be the number of nonzero entries in $\bm{\nu}$, where $\supp(\bm{\nu}) =  \lbrace k: \nu_k \neq 0 \rbrace \subseteq \bbN_{{0}}$ is the support of $\bnu$. We now  define the  set of multi-indices with at most finitely-many nonzero entries by 
\begin{equation*}
\cF := \lbrace \bm{\nu} = (\nu_k)_{k \in \bbN} \in  \bbN_0^{\bbN}: \|\bm{\nu}\|_0  < \infty \rbrace.
\end{equation*}
Let $\varrho$ be the uniform probability measure on $\cU=[-1,1]^{\bbN}$ and $\lbrace \Psi_{\bm{\nu}} \rbrace_{\bnu \in \cF}  $ be the orthonormal Legendre basis of $L^2_{\varrho}(\cU)$ constructed via the tensorization 
\begin{equation*}
\Psi_{\bm{\nu}}(\bm{y}) = \prod_{k \in \bbN} \Psi_{\nu_k}(y_k), \quad \bm{y} \in \cU, \bm{\nu} \in \cF,
\end{equation*}
where $\Psi_{\nu}$ is the univariate, orthonormal Legendre polynomial of degree $\nu$. Then any $f \in L^2_{\varrho}(\cU;\cV)$ has the convergent expansion
\begin{equation}
\label{f-exp}
f = \sum_{\bm{\nu} \in \cF} c_{\bm{\nu}} \Psi_{\bm{\nu}},\qquad \text{where }c_{\bm{\nu}} = \int_{\cU} f(\y) \Psi_{\bm{\nu}}(\bm{y}) \D \varrho(\bm{y}) \in \cV.
\end{equation}
Now let $S \subset  \cF$ be a finite index set of size $N$. Then the truncated series of $f$ is given by
\begin{equation}\label{def_f_S}
f_S = \sum_{\bm{\nu} \in S} c_{\bm{\nu}} \Psi_{\bm{\nu}}.
\end{equation}

\subsection{$\ell^p$-summability and best $s$-term rates}

Given a (multi-)index set $\Lambda$, we now define the $\ell^p(\Lambda;\cV)$-norm by
\begin{equation*}
\|\bm{v}\|_{p;\cV} =
\begin{cases}
\left(\sum_{\bnu \in \Lambda}  \|v_{\bnu}\|_{\cV}^p\right)^{1/p}, &\quad 1 \leq p < \infty, \\
\sup_{\bnu \in \Lambda}\|v_{\bnu}\|_{\cV}, &\quad   p = \infty,  \\
\end{cases}\qquad \bm{v} = (v_{\bnu})_{\bnu \in \Lambda}.
\end{equation*}
We shall typically use this in the case $\Lambda = \cF$ or $\Lambda = [N]$.

The proof of the following lemma can be found in \cite[Thm.~3.28]{adcock2022sparse} and provides a key summability estimate for the Legendre coefficients of a $(\bm{b},1)$-holomorphic function.

\begin{lemma}\label{lem:coeff_known}
Let $0<p<1$ and $\bm{b} \in \ell^p (\bbN)$ with $\bm{b} \geq {\bm{0}} $. Then the Legendre coefficients  $\bm{c} = ({c}_{\bnu})_{\bnu \in \cF}$ in \eqref{f-exp} satisfy
\begin{equation}\label{eq:coeff_known}
\|\bm{c}\|_{p;\cV} \leq C(\bm{b},p), \quad \forall f \in \cH(\bm{b}),
\end{equation}
where $C(\bm{b},p)$ depends on $\bm{b}$ and $p$ only.
\end{lemma}

We are particularly interested in bounding the supremum of $C = C(\b,q)$  over  $\bm{b} \in \ell^{p}_{\mathsf{M}}(\bbN)$ for $0 < p < q < 1$. This bound is used in the proof of part (b) of Theorems~\ref{thm:unknown-upper-1} and \ref{thm:unknown-upper-2}. The following result is obtained by modifying the proof of \cite[Thm~3.28]{adcock2022sparse}.
\begin{lemma}\label{cor:supr}
Let $0<p<q<1$. Then
\begin{equation*}
\sup_{\|\bm{b}\|_{p,\mathsf{M}} \leq 1} C(\b,q) \leq c_{p,q},
\end{equation*}
where  $C=C(\b,q)$ is the constant in \eqref{eq:coeff_known} and  $c_{p,q}$ is {a} positive constant    depending on $p$ and $q$ only. 
\end{lemma}
\begin{proof}
Consider $\bm{b} \in \ell^p_{\mathsf{M}}(\bbN)$ with $\|\bm{b}\|_{p,\mathsf{M}} \leq 1$. Notice   from \cite[Eq.~(3.48)]{adcock2022sparse}   that the constant   $C$ in \eqref{eq:coeff_known} can be taken to be
 \begin{equation}\label{eq:Constant_C}
 C(\b,q) = \xi(\tilde{\kappa})^d \left( \sum_{n=0}^{\infty} \dfrac{(2n+1)^{{q}/2}}{\tilde{\kappa}^{qn}} \right)^{d/q} \|\bm{g}(\b)\|_q,
 \end{equation}
 where $\xi(t)=\min \lbrace 2t,\frac{\pi}{2}(t+t^{-1})\rbrace/(t-1)$ for every $t>1$, the term $\tilde{\kappa}= \tilde{\kappa}(\b)>1$ is defined as the unique solution to 
 \begin{equation*}
 \dfrac{\tilde{\kappa}+\tilde{\kappa}^{-1}}{2} =1+\dfrac{1}{2\|\bm{b}\|_1},
 \end{equation*}
 and 
 \begin{equation}\label{eq:h}
\begin{split}
 g(\b)_{\bnu} &= \dfrac{\|\bnu\|_1!}{\bnu!} \bm{h}(\b)^{\bnu} \prod_{j \in \bbN} (\xi(\tilde{\kappa}) \sqrt{3 \nu_j}+1), \quad \forall \bnu \in \cF,\\
 {h} (\bm{b})_j &= 2\E b_{j+d},
\end{split}
 \end{equation}
 where $d\in \bbN$ is a truncation parameter.
 
We aim to bound \eqref{eq:Constant_C} by a constant  depending on $p$ and $q$. First,  we show that there is a convergent sequence $\tilde{\bm{h}}$ independent  of $\bm{b}$ that can replace  $\bm{h}(\b)$ in \eqref{eq:h}.  Then, we proceed using similar arguments to those in the proof of \cite[Thm.~3.28]{adcock2022sparse} to get the result.

Let  $\tilde{\bm{b}}$ be the minimal monotone majorant \R{min-mon-maj} of $\bm{b}$. Using Stechkin's inequality we get
\bes{
 \sum_{j = n+1}^\infty \tilde{b}_j = \sigma_n(\tilde{\bm{b}})_1 \leq n^{1-1/p} \nmu{\tilde{\bm{b}}}_{p} =   n^{1-1/p} \nm{\bm{b}}_{p,\mathsf{M}} \leq c_p n^{1-1/p},
}
{for some constant $c_p$ depending only on $p$.}
Also by monotonicity,
\bes{
n \tilde{b}_{2n} \leq \tilde{b}_{n+1} + \cdots + \tilde{b}_{2n} \leq \sum_{j = n+1}^\infty \tilde{b}_j \leq c_p n^{1-1/p}.
}
Hence $b_{n} \leq \tilde{b}_n \leq \tilde{c}_p n^{-1/p}$ for a possible different constant depending on $p$. Note  that the inequality for odd values of $n$ can be established using a similar argument.  Keeping this in mind, we define the sequence $\tilde{\bm{h}}(p)=(\tilde{{h}}(p)_j)_{j \in \bbN}$  by
\begin{equation*}
\tilde{{h}}(p)_j = {2 \tilde{c}_p \E (j+d)^{-1/p}},
\end{equation*}
where $d\in \bbN$ is  a parameter that will be chosen in the next step.
Thus, we get the bound 
\begin{equation*}
 {h}(\bm{b})_j = {2 \E b_{j+d}}{} \leq 2 \tilde{c}_p \E (j+d)^{-1/p} = \tilde{{h}}(p)_j, \quad \forall j \in \bbN.
\end{equation*}
Observe that  $\tilde{\bm{h}}(p) \in \ell^1(\bbN)$. Moreover,  using the fact that $q/p>1$ and a simple convergence argument, we obtain that  
\begin{equation}\label{eq:step1}
\| \tilde{\bm{h}}(p) \|_{q} =   2\tilde{c}_p \E  \left(\sum_{j \in \bbN} (j+d)^{-q/p} \right)^{1/q} < \infty,
\end{equation}
 which implies that $\tilde{\bm{h}}(p)  \in \ell^q(\bbN)$.  We now choose  $d=d({p})$ as the minimum $d \in \bbN$ such that
\begin{equation}\label{eq:step2}
\nmu{\tilde{\bm{h}}(p)}_1 = \sum_{j \in \bbN} {2 \tilde{c}_p  \E (j+d)^{-1/p}}{} < 1.
\end{equation}
On the other hand, since  $\|\bm{b}\|_1 \leq  \|\bm{b}\|_{1,\mathsf{M}} \leq \|\bm{b}\|_{p,\mathsf{M}} \leq 1 $  and examining the solution of equation
 \begin{equation*}
 \dfrac{\tilde{\kappa}+\tilde{\kappa}^{-1}}{2} =1+\dfrac{1}{2\|\bm{b}\|_1}, 
 \end{equation*}
we deduce that $\tilde{\kappa} > 2.6$ through a straightforward inspection. Hence, from the definition of $\xi$ and the lower bound on $\tilde{\kappa}$ we get that $\xi(\tilde{\kappa}) \leq 2 \tilde{\kappa}/(\tilde{\kappa}-1) \leq 4$. Note that this upper bound is independent  of $\tilde{\kappa}$, and therefore independent of $\bm{b}$. Keeping this in mind, we get  
\begin{equation}\label{eq:gtilde}
 g(\b)_{\bnu}  = \dfrac{\|\bnu\|_1!}{\bnu!} \bm{h}(\b)^{\bnu} \prod_{j \in \bbN} (\xi(\tilde{\kappa}) \sqrt{3 \nu_j}+1) 
 \leq
 \dfrac{\|\bnu\|_1!}{\bnu!} \tilde{\bm{h}}(p)^{\bnu} \prod_{j \in \bbN} (4 \sqrt{3 \nu_j}+1)  =:\tilde{g}(p)_{\bnu} 
 , \quad \forall \bnu \in \cF,
\end{equation}
which implies that $\|\bm{g}(\b)\|_q \leq \|\tilde{\bm{g}}(p)\|_q$. Therefore,  we can bound \eqref{eq:Constant_C}  by 
 \begin{equation}\label{eq:Constant_C2} 
 C(\b,q) = \xi(\tilde{\kappa})^d \left( \sum_{n=0}^{\infty} \dfrac{(2n+1)^{q/2}}{\tilde{\kappa}^{qn}} \right)^{d/q} \|\bm{g}(\b)\|_q
 \leq \xi(\tilde{\kappa})^d \left( \sum_{n=0}^{\infty} \dfrac{(2n+1)^{q/2}}{\tilde{\kappa}^{qn}} \right)^{d/q} \|\tilde{\bm{g}}(p)\|_q. 
 \end{equation}
 To show that $\|\tilde{\bm{g}}(p)\|_q < \infty$   we combine  \eqref{eq:step1} with \eqref{eq:step2} and apply  \cite[Lem.~3.29]{adcock2022sparse}.   It remains to bound the other term  in the previous inequality. Recall that $\xi(\tilde{\kappa})  \leq 4$ and $\tilde{\kappa} > 2.6$. Then,
\begin{equation*}
\begin{split}
\xi(\tilde{\kappa})^d \left( \sum_{n=0}^{\infty} \dfrac{(2n+1)^{q/2}}{\tilde{\kappa}^{qn}} \right)^{d/q} 
\leq    4^d \left( \sum_{n=0}^{\infty} \dfrac{(2n+1)^{q/2}}{(2.6)^{qn}} \right)^{d/q}  \leq \overline{c}_{p,q}< \infty,\\
\end{split}
\end{equation*}
where  $\overline{c}_{p,q}$ is a positive constant depending on $p$ and $q$ only. {Note that $\overline{c}_{p,q}$ depends on $p$ due to the dependence of $d$ on $p$.  }
In this way, by taking supremum over $\|\bm{b}\|_{p,\mathsf{M}} \leq 1$ in  \eqref{eq:Constant_C2} we  get  the result.
\end{proof}
We now present a best $s$-term approximation rate for the Legendre polynomials. As in \S\ref{S:main_res}, for nonsparse vectors in $\ell^p(\cF;\cV)$, we define the $\ell^p$-norm best $s$-term approximation error as 
\begin{equation} \label{def:sigmaV}
\sigma_s(\bm{x})_{p;\cV} = \inf_{\bm{z} \in \ell^p(\cF;\cV)} \{ \| \bm{x}-\bm{z}\|_{p;\cV}: |\supp (\bm{z})| \leq s \}, \quad \bm{x} \in \ell^p(\cF;\cV), 
\end{equation}
where  $$\supp(\bm{z}) = \{ \bm{\nu} \in \cF: \|{z}_{\bnu} \|_{\cV} \neq 0  \}$$
is the support of the vector $\bm{z}$.
The following result can be deduced from Stechkin's inequality and Lemma~\ref{lem:coeff_known}. Note that the proof of  Lemma~\ref{lem:coeff_known} (see \cite[Thm.~3.28]{adcock2022sparse}) involves establishing the summability of a bounding sequence for the $\cV$-norms of the polynomial coefficients in \eqref{f-exp}. {This bound is equal to $\|f\|_{L^{\infty}(\cR(\bm{b});\cV ) }$ multiplied by a factor that is independent of $f$ and depending on $\bm{b}$ only.} Consequently, the index set in the following result is independent of $f$. 

\begin{corollary}\label{cor:known}
Let $0 <p <1$, $q \geq p$, $\bm{b} \in \ell^p(\bbN)$ with $\bm{b} \geq {\bm{0}}$ and $s \in \bbN$. Then, there exists a set $S \subset \cF$ of size $|S| \leq s$ depending on $\bm{b}$ and $p$ only such that 
\begin{equation}
\sigma_s (\bm{c})_{q;\cV} \leq \|\bm{c}-\bm{c}_{S}\|_{q;\cV}  \leq C(\b,p) \cdot s^{1/q-1/p}, \quad \forall f \in \cH (\bm{b}),
\end{equation}
where  $C(\b,p)>0$    is the constant in \eqref{eq:coeff_known} and $\bm{c}=(c_{\bnu})_{\bnu \in \cF}$ are the Legendre coefficients in \eqref{f-exp}. 
\end{corollary}

\subsection{$\ell^p_{\mathsf{A}}$-summability and best $s$-term rates in anchored sets}\label{Ap:anch}

In the last part of this appendix we require the notion of lower and anchored sets (see, e.g., \cite[\S~3.9]{adcock2022sparse}). A set $\Lambda \subseteq \cF$ is \textit{lower} if $\bm{\nu} \in \Lambda$  and $\bm{\mu} \leq \bm{\nu} $ implies that $\bm{\mu} \in \Lambda$  for every $\bm{\nu},\bm{\mu} \in \cF$. Moreover, a set $\Lambda \subseteq \cF$ is \textit{anchored} if it is lower and if  $\bm{e}_j \in \Lambda$  implies that $\{\bm{e}_1,\bm{e}_2, \ldots, \bm{e}_{j} \}\subseteq \Lambda$ for every $j \in \mathbb{N}$.

We also introduce the \textit{minimal anchored majorant} of a sequence and the $\ell^p_{\mathsf{A}}$ space. For further details we refer to  \cite[Def.~3.31]{adcock2022sparse}. Let $0<p< \infty$. A sequence $\bm{c} \in \ell^{\infty}(\cF;\cV)$ belongs to 
$\ell^p_{\mathsf{A}}(\cF;\cV)$  if its minimal anchored majorant $\tilde{\bm{c}} = (\tilde{c}_{\bm{\nu}})_{\bm{\nu} \in \cV}$, defined by
\begin{equation}\label{def:min_anch}
\tilde{c}_{\bnu} = 
\begin{cases}
\sup \{ \|c_{\bm{\mu}}\|_{\cV}: \bm{\mu} \geq \bnu \} & \text{ if }  \bnu \neq \bm{e}_j \text{ for any } j \in \bbN,  \\ 
\sup \{ \|c_{\bm{\mu}}\|_{\cV}: \bm{\mu} \geq \bm{e}_i \text{ for some } i \geq j \} & \text{ if } \bnu =\bm{e}_j \text{ for some } j \in \bbN, 
\end{cases}
\end{equation}
belongs to $\ell^p(\cF)$. In particular, we define its  $\ell^p_{\mathsf{A}}(\cF;\cV)$-norm by 
\begin{equation}
\|\bm{c}\|_{p,\mathsf{A};\cV} = \|\tilde{\bm{c}}\|_{p;\cV}.
\end{equation}

Similar to Lemma \ref{lem:coeff_known}, the following result shows summability of the Legendre coefficients of a $(\bm{b},1)$-holomorphic function in the $\ell^p_{\mathsf{A}}$-norm.  

\begin{lemma}\label{lem:coeff_known_anch}
Let $0<p<1$ and $\bm{b} \in \ell^p_{\mathsf{M}} (\bbN)$ with $\bm{b} \geq {\bm{0}} $. Then the Legendre coefficients  $\bm{c} = ({c}_{\bnu})_{\bnu \in \cF}$ in \eqref{f-exp} satisfy
\begin{equation}\label{eq:coeff_known_anch}
\|\bm{c}\|_{p,\mathsf{A};\cV} \leq C_{\mathsf{A}}(\bm{{b}},p), \quad \forall f \in \cH(\bm{{b}}),
\end{equation}
where   $C_{\mathsf{A}}(\bm{{b}},p)$ depends on $\bm{b}$ and $p$ only.
\end{lemma}
\begin{proof}
Let $\tilde{\bm{b}}$ be the minimal monotone majorant of $\bm{b}$, defined in \eqref{min-mon-maj}. Following the same argument as in \cite[Cor.~8.2]{adcock2022nearoptimal} we get that $\cR(\tilde{\bm{b}}) \subseteq \cR(\bm{{b}})$ where $\cR(\bm{b})$ is as in \eqref{def:b-eps-holo}. Therefore, $\cH(\bm{b}) \subseteq \cH(\tilde{\bm{b}})$. Since  $\tilde{\bm{b}} \in \ell^p(\bbN)$ is monotonically nonincreasing  \cite[Thm.~3.33]{adcock2022sparse}  implies the result.
\end{proof}

As in Lemma \ref{cor:supr}, we are interested in bounding the supremum of $C_{\mathsf{A}} = C_{\mathsf{A}}(\b,q)$  over  $\bm{b} \in \ell^{p}_{\mathsf{M}}(\bbN)$. This bound will be useful  for the proof of part (b) of Theorem \ref{thm:unknown-upper-2}. The following result is obtained by modifying the proof of \cite[Thm~3.33]{adcock2022sparse}.
\begin{lemma}\label{cor:supr_Anchored}
Let $0<p<q<1$. Then
\begin{equation*}
\sup_{\|\bm{b}\|_{p,\mathsf{M}} \leq 1} C_{\mathsf{A}}(\bm{{b}},q) \leq c_{p,q},
\end{equation*}
where $C_{\mathsf{A}}(\bm{{b}},p)$ is the constant in \eqref{eq:coeff_known_anch} and  $c_{p,q}$ is positive constant    depending on $p$ and $q$ only. 
\end{lemma}

\begin{proof}
Let    $\tilde{\bm{b}}$ be the minimal monotone majorant \R{min-mon-maj} of $\bm{b} \in \ell^p_{\mathsf{M}} (\bbN)$ with $\nmu{\tilde{\bm{b}}}_p =\|\bm{b}\|_{p,\mathsf{M}} \leq 1$ and  $\tilde{\kappa}= \tilde{\kappa}(\tilde{\bm{b}})>1$ be the  unique solution to 
 \begin{equation*}
 \dfrac{\tilde{\kappa}+\tilde{\kappa}^{-1}}{2} =1+\dfrac{1}{4\|\tilde{\bm{b}}\|_1}.
 \end{equation*}
 Observe that $\tilde{\kappa}\geq 2$. Then,  a simple inspection reveals that
\begin{equation}
 \sqrt{2n+1} \leq  \dfrac{6}{5}  \left( \dfrac{3}{2} \right)^n \leq \dfrac{6}{5}  \left( \dfrac{1+\tilde{\kappa}}{2}\right)^n, \forall n \in \bbN.
 \end{equation} 
Also, define $\eta= \eta (\tilde{\kappa}):= (1+\tilde{\kappa})/(2\tilde{\kappa})<1$.
 Now, notice from \cite[Eq.~(3.62)]{adcock2022sparse} and the last paragraph  in the proof of \cite[Thm.~3.33]{adcock2022sparse} ,  that the constant   $C_{\mathsf{A}}$ in \eqref{eq:coeff_known_anch} can be taken to be  
 \begin{equation}\label{eq:Constant_C_anch}
C_{\mathsf{A}}(\bm{{b}},q) = \widetilde{C}_{\mathsf{A}}(\tilde{\bm{b}},q)  = D_1^d \left( \sum_{n=0}^{\infty} \eta^{qn} \right)^{d/q} \|\bm{g}(\tilde{\bm{b}})\|_q,
 \end{equation}
 where $D_1 =D_1(\tilde{\kappa}):= \max \{ 1,6/5\xi(\tilde{\kappa})\}$ with  $\xi(t)=\min \lbrace 2t,\frac{\pi}{2}(t+t^{-1})\rbrace/(t-1)$ for every $t>1$, 
and 
 \begin{equation}\label{eq:h_anch}
\begin{split}
 g(\tilde{\bm{b}})_{\bnu} &= \dfrac{\|\bnu\|_1!}{\bnu!} \bm{h}(\tilde{\bm{b}})^{\bnu} \prod_{j \in \bbN} (\xi(\tilde{\kappa}) \sqrt{3 \nu_j}+1), \quad \forall \bnu \in \cF,\\
 {h} (\tilde{\bm{b}})_j &= 4\E \tilde{{b}}_{j+d},
\end{split}
 \end{equation}
 where $d\in \bbN$ is a truncation parameter.
 
 Now, following the same arguments  as in \eqref{eq:h}--\eqref{eq:Constant_C2} we deduce that 
 \begin{equation}\label{eq:Constant_C_anch2}
\widetilde{C}_{\mathsf{A}}(\tilde{\bm{b}},q)  = D_1^d \left( \sum_{n=0}^{\infty} \eta^{qn} \right)^{d/q} \|\bm{g}(\tilde{\bm{b}})\|_q \leq  D_1^d \left( \sum_{n=0}^{\infty} \eta^{qn} \right)^{d/q} \|\tilde{\bm{g}}(p)\|_q,
 \end{equation} 
 where, using \cite[Lem.~3.29]{adcock2022sparse} once more, we see that the sequence $\tilde{\bm{g}}(p)$, defined in \eqref{eq:gtilde}  satisfies   $\|\tilde{\bm{g}}(p)\|_q< \infty$. It remains to bound the other term  in the previous inequality. As in the last steps in the proof of Lemma \ref{cor:supr}, with  $\xi (\tilde{\kappa}) \leq 4$ and $\tilde{\kappa} \geq 2$, we deduce that

\begin{equation}
  D_1^d \left( \sum_{n=0}^{\infty} \eta^{qn} \right)^{d/q}  \leq    4^d\left( \dfrac{6}{5} \right)^d \left( \sum_{n=0}^{\infty} \left( \dfrac{4}{5}\right)^{qn} \right)^{d/q}   \leq  c_{p,q} < \infty,
 \end{equation} 
 where $ c_{p,q}$ is a positive constant depending on $p$ and $q$ only. Finally, by taking the supremum over $\|\bm{b}\|_{p,\mathsf{M}} \leq 1$ in  \eqref{eq:Constant_C_anch2} we  get  the result.
\end{proof}

Let $0 < p \leq \infty$. We now introduce the concept of the \textit{$\ell^p$-norm best $s$-term approximation error in anchored sets}. This is defined as
\begin{equation} \label{def:sigma_anch}
\sigma_{s,\mathsf{A}}(\bm{x})_{p;\cV} = \inf_{\bm{z}   \in \ell^p(\cF; \cV)} \{ \| \bm{x}-\bm{z}\|_{p;\cV}: |\supp (\bm{z})| \leq s, \supp(\bm{z}) \text{ anchored} \}, \quad \bm{x} \in \ell^p ({\cF;\cV}).
\end{equation}
Now we provide an algebraic $s$-term rate in anchored sets. The following result follows similar arguments to those used in Corollary \ref{cor:known} and it is deduced by  applying \cite[Lemma 3.32]{adcock2022sparse} and Lemma \ref{lem:coeff_known_anch} to the Legendre coefficients $\bm{c}=({c}_{\bnu})_{\bnu \in \cF}$ in \eqref{f-exp}. Observe that, to prove Lemma \ref{lem:coeff_known_anch}  the $\cV$-norm of the coefficients in \eqref{f-exp} are bounded by a monotonically nonincreasing sequence (see \cite[Eq.~(3.55)]{adcock2022sparse})  that only depends on  $\bm{b}$. Therefore, the anchored set in the following corollary is independent of $f$.
\begin{corollary}\label{cor:anchored_sigma}
Let $0 <p <1$, $q \geq p$, $\bm{b} \in \ell^p_{\mathsf{M}}(\bbN)$ and $s \in \bbN$. Then, there exists an anchored set $S \subset \cF$ of size $|S| \leq s$ such that
\begin{equation}
\sigma_{s,\mathsf{A}} (\bm{c})_{q;\cV}\leq \|\bm{c}-\bm{c}_{S}\|_{q;\cV}  \leq C_{\mathsf{A}}(\bm{{b}},p) \cdot s^{1/q-1/p}, \quad \forall f \in \cH (\bm{b}),
\end{equation}
where $C_{\mathsf{A}}({\bm{b}},p)$ is the constant in \eqref{eq:coeff_known_anch}  and $\bm{c}=(c_{\bnu})_{\bnu \in \cF}$ are the Legendre coefficients in \eqref{f-exp}.  
\end{corollary}

\section{Widths of weighted $\ell^p$-norm balls in $\bbR^N$}\label{s:widths}

In this appendix, we present a proposition providing a lower bound to the Gelfand $m$-width in terms of $m$ and $N \in \bbN$, an equality theorem proving the connection between Gelfand $m$-widths and Kolmogorov widths for a particular set of spaces, {a duality  result by Stesin \cite{stesin1975aleksandrov}}  and a lemma that establishes the relationship between the unit balls in these spaces using the Kolmogorov $m$-widths.

First, recall the definition of the Gelfand $m$-width from \S\ref{S:apxA}. The following proposition  can be obtained from  \cite[Prop.\ 2.1]{foucart2010gelfand} by an inspection of the proof. 
\begin{proposition}[Lower bound]\label{prop:lowerbound}
Let $N\in \bbN$. For $0 < p \leq 1$, {$m<N$} and $p < q \leq \infty$,
\begin{equation}\label{eq:lowerbound}
d^m(B^p_N , \ell^q_N)  \geq \left( \frac12 \right)^{\frac2p-\frac1q}  \min \left \{ 1 , \frac{\frac{2p}{\log(3^8\E)}\log(\E N / m)}{m} \right \}^{\frac1p-\frac1q}  .
\end{equation}
\end{proposition}
We now give the proof of an equality result based on the methodology described in \cite[Lem.\ 10.15]{foucart2013mathematical}. Specifically, we show that the Kolmogorov widths of $\ell_N^p$-balls in the weighted $\ell^q_N$ space are equivalent to certain Gelfand widths.

First, for $1 \leq p,p^*,q,q^* \leq  \infty$ we recall the definitions of the Gelfand and Kolmogorov widths for this particular case, see \S\ref{S:apxA}. The Gelfand $m$-width of the subset $B^{q^*}_N(1/\bm{w}) $ of $\ell^{p^*}_N $ is
\begin{equation*}
d^m(B^{q^*}_N(1/\bm{w}),\ell^{p^*}_N ) = \inf \left \{ \sup_{\bm{x}\in B^{q^*}_N(1/\bm{w}) \cap \cX^m} \nm{\bm{x}}_{p^*},\ \text{$\cX^m$ a subspace of $\ell^{p^*}_N $ with $\mathrm{codim}(\cX^m) \leq m$} \right \},
\end{equation*}
and the   Kolmogorov $m$-width of a subset $B^p_N$ of the space $\ell^q_N(1/\w)$ is
\begin{equation*}
d_m(B^p_N,\ell^q_N(\w)) = \inf \left \{ \sup_{\bm{x} \in B^p_N} \inf_{\bm{z} \in \cX_m} \nm{\bm{x} - \bm{z}}_{q,\w},\text{ $\cX_m$ a subspace of $\ell^q_N(\w)$ with $\dim(\cX_m) \leq m$} \right \}.
\end{equation*}

\thm{
[Stesin]
\label{thm:stesin}
Let {$N \in \bbN$ with $N>m$}, $1 \leq q < p \leq \infty$,  and $\bm{w} \in \bbR^N$ be a vector of positive weights. Then
\bes{
d_m(B^p_N(\bm{w}),\ell_N^q) = \left ( \max_{\substack{i_1,\ldots,i_{N-m} {\in [N]}\\ i_k \neq i_j}} \left ( \sum^{N-m}_{j=1} w^{pq/(p-q)}_{i_j} \right )^{1/p-1/q} \right )^{-1}.
}
}

\begin{theorem}[Equality]\label{thm:equal_dm}
For $1 \leq p , q \leq \infty$, let  $\bm{w} \in \bbR^N$ be a vector of positive weights and $p^*,q^*$ be such that $1/p^*+1/p=1$ and  $1/q^*+1/q=1$. Then
\begin{equation}
d_m(B^p_N , \ell^q_N(\bm{w}) ) = d^m ( B^{q^*}_N(1/\bm{w}) , \ell^{p^*}_N ) .
\end{equation}
\end{theorem}

\begin{proof}
First, given  $\bm{x} \in B_N^p$ and  a subspace $X_m$ of $\ell^q_N(\w)$, we follow the same arguments as those in \cite[Lem.\ 10.15]{foucart2013mathematical} to obtain
\begin{equation*}
\inf_{\bm{z} \in X_m} \|\bm{x}-\bm{z}\|_{q,\w} =   \ip{\phi}{\bm{x}},
\end{equation*}
for some linear bounded functional $\phi \in  { X_m^{\circ} }$, with $\|\phi\|_{(\ell^{q}_{{N}}(\w))^*} \leq 1$, where 
\begin{equation*}
 X_m^{\circ} := \{\phi \in (\ell^q_N(\w))^*: \phi(\bm{x})=0, \quad \forall \bm{x} \in X_m \}.
\end{equation*}
Now, by the definition in \eqref{eq:defB_weighted}  we  get
\begin{equation*}
\inf_{\bm{z} \in X_m} \|\bm{x}-\bm{z}\|_{q,\w}  \leq \sup_{\substack{\phi \in B_N^{q*}(1/\w) \cap X_m^{\circ} }}  \ip{\phi}{\bm{x}}.
\end{equation*}
On the other hand, for all  $\phi \in B_N^{q*}(1/\w)  \cap X_m^{\circ}$,  and $\bm{z} \in X_m$ we have 
\begin{equation*}
\ip{\phi}{\bm{x}} = \ip{\phi}{\bm{x}-\bm{z}} \leq  \|\phi\|_{(\ell^{q}_N(\w))^* } \|\bm{x}-\bm{z}\|_{q,\w}.
\end{equation*} 
Then we deduce the following equality
\begin{equation}
\inf_{\bm{z}\in X_m} \|\bm{x}-\bm{z}\|_{q,\w}  = \sup_{\substack{\phi  \in B_N^{q*}(1/\w)  \cap X_m^{\circ} }}  \ip{\phi}{\bm{x}}.
\end{equation}
Taking  supremum on both sides over $\bm{x}   \in B_N^p$, we have
\begin{align*}
 \sup_{\bm{x}  \in B_N^p} 
\inf_{\bm{z} \in X_m}  \|\bm{x}-\bm{z}\|_{q}  &
=\sup_{\bm{x}    \in B_N^p}
\sup_{\substack{\phi  \in B_N^{q*}(1/\w)  \cap X_m^{\circ} }}  \ip{\phi}{\bm{x}}  \\
&=\sup_{\substack{\phi  \in B_N^{q*}(1/\w)  \cap X_m^{\circ} }} 
\sup_{\bm{x}    \in B_N^p}
 \ip{\phi}{\bm{x}} \\
  &= \sup_{\substack{\phi  \in B_N^{q*}(1/\w)  \cap X_m^{\circ} }} 
 \|\phi\|_{p^*}  { \sup_{\bm{x}    \in B_N^p} } \|\bm{x}\|_{p}  \\ 
 &
  = \sup_{\substack{\phi  \in B_N^{q*}(1/\w)  \cap X_m^{\circ} }} 
 \|\phi\|_{p^* }.
\end{align*}
Taking the infimum over all subspaces $X_m$ with $\dim(X_m) \leq m$ and noticing the one-to-one correspondence between the subpaces $X_m^{\circ}$ and the subspaces $\cX^m$ with $\text{codim}(\cX^m) \leq m$, we obtain
\begin{equation*}
d_m(B^p_N , \ell^q_N(\bm{w}) ) = d^m ( B^{q^*}_N(1/\bm{w}) , \ell^{p^*}_N ),
\end{equation*}
as required.
\end{proof}
The following result establishes a connection between the unit ball in the weighted space $\ell^q_N(\w)$ and the unit weighted ball in $\ell^q_N$, using the Kolmogorov $m$-width  
\begin{equation*}
d_m(B^p_N(1/\w),\ell^q_N) = \inf \left \{ \sup_{\bm{x} \in B^p_N(1/\w)} \inf_{\bm{z} \in \cX_m} \nm{\bm{x} - \bm{z}}_{q},\text{ $\cX_m$ a subspace of $\ell^q_N$ with $\dim(\cX_m) \leq m$} \right \}.
\end{equation*}
\begin{lemma}\label{lem:widths_2}
Let $\bm{w} \in \bbR^N$ be a vector of positive weights and $1 \leq p,q\leq \infty$. Then
\begin{equation}
d_m(B^p_N,\ell^q_N(\w)) = d_m(B^p_N(1/\w), \ell^q_N).
\end{equation}
\end{lemma}
\begin{proof}
Let $\bm{x} \in B^p_N$ and $\bm{z} \in X_m$, where $X_m$ is a $m$-dimensional subspace of $X = \ell^q_N(\w)$. Then
\begin{align*}
\inf_{\bm{z} \in X_m} \nm{\bm{x} - \bm{z}}_{q,\w}= \inf_{\bm{z} \in X_m} \left(\sum_{i \in [N]} \left ( \frac{|x_i - z_i |}{w_i} \right )^{q} \right)^{1/q} =   \inf_{\bm{z}' \in X'_m} \nm{(1/\w) \odot \bm{x} - \bm{z}'}_{q}.
\end{align*}
Notice there is a one-to-one correspondence between subspaces $X_m$ and subspaces $X'_m = \{ (1/\w) \odot \bm{z} : \bm{z} \in X_m \}$. Also, there is a one-to-one correspondence between $\bm{x} \in B^p_N$ and $(1/\w) \odot \bm{x} \in B^p_N(1/\w)$. Thus,
\bes{
d_m(B^p_N,\ell^q_N(\w)) = d_m(B^p_N(1/\w), \ell^q_N).
}
as required.
\end{proof}

\small
\bibliographystyle{plain}
\bibliography{arxivsub}

\end{document}